\documentclass[11pt]{article}
\usepackage{geometry}                
\usepackage[parfill]{parskip}    
\usepackage{graphicx,color}
\usepackage{amssymb,amsmath,amsthm}
\usepackage{epstopdf}
\usepackage{wrapfig}

\usepackage{tikz}

\usepackage{marvosym}
\usepackage{tikz}
\usepackage{verbatim}

\usepackage[pdftex,bookmarks,colorlinks,breaklinks]{hyperref}
\definecolor{dullmagenta}{rgb}{0.4,0,0.4}   
\definecolor{darkblue}{rgb}{0,0,0.4}
\hypersetup{linkcolor=red,citecolor=blue,filecolor=dullmagenta,urlcolor=darkblue}

\def\red#1{\textcolor[rgb]{0.9, 0, 0}{#1}}

\usepackage{cancel}

\newcommand{\ii}{{\rm i}}
\newcommand{\ee}{{\rm e}}
\newcommand{\dd}{{\rm d}}

\newcommand{\CC}{{\mathbb C}}
\newcommand{\RR}{{\mathbb R}}

\newcommand{\DD}{{\mathbb D}}
\newcommand{\TT}{{\mathbb T}}

\def\conn{\mathop{\rm conn}}
\def\clos{\mathop{\rm clos}}
\def\inte{\mathop{\rm int}}

\newtheorem{thm}{Theorem}[section] 
\newtheorem{prop}[thm]{Proposition}
\newtheorem{lemma}[thm]{Lemma}
\newtheorem{cor}[thm]{Corollary}
\newtheorem*{thm-others}{Theorem}

\newtheorem{innercustomthm}{Theorem}
\newenvironment{customthm}[1]
  {\renewcommand\theinnercustomthm{#1}\innercustomthm}
  {\endinnercustomthm}

\theoremstyle{remark}
\newtheorem*{defn}{Definition}

\newtheorem*{remark}{Remark}

\def\C{{\mathbf{C}}}
\def\R{{\mathbb R}}

\def\bs{\bigskip}

\definecolor{MyGray}{rgb}{0.92,0.96,0.98}


\definecolor{MyRed}{rgb}{0.95,0.8,0.85}


\title{Sharpness of connectivity bounds for quadrature domains}
\author{Seung Yeop Lee and Nikolai Makarov}

\begin{document}

\maketitle

\section{Introduction and results}\label{sec:1}

In this paper we prove the sharpness of connectivity bounds established in \cite{1stpaper}. The proof depends on some
facts in the theory of univalent polynomials. We also discuss applications to the equation $r(z) =\bar z$ where $r$ is a rational function.

\subsection{Quadrature domains}
By definition, a bounded  connected open set  $\Omega\subset \CC$, satisfying $\inte\clos\Omega=\Omega$, is a {\em bounded quadrature domain} (BQD) if there  exists a rational function $r=r_\Omega$, which we call the {\em quadrature function} of $\Omega$, such that $r(\infty)=0$, all poles of $r$ are in $\Omega$, and
\begin{equation}\label{eq:QI}
\forall f\in C_A(\Omega), \qquad
\int_\Omega f\,\dd A=\frac{1}{2\ii}\oint_{\partial\Omega} f(z)\,r(z)\,\dd z,
\end{equation}
where $C_A(\Omega)$ is the space of analytic functions on $\Omega$ continuous up to the boundary.

Similarly, a domain $\Omega\subset \widehat\CC$, satisfying $\infty\in\Omega$ and $\inte\clos\Omega=\Omega$, is called an {\em unbounded quadrature domain} (UQD) if there  exists a rational function $r=r_\Omega$ such that all poles of $r$ are in $\Omega$ and
\begin{equation*}
f\in C_A(\Omega), ~~ f(\infty)=0  \quad\Longrightarrow\quad
\int_{\Omega} f\,\dd A=\frac{1}{2\ii}\oint_{\partial\Omega} f(z)\,r(z)\,\dd z,
\end{equation*}
where the integral over $\Omega$ is understood in the principal value sense.

For a quadrature domain (bounded or unbounded) $\Omega$, 
we define $$d_\Omega=\deg r_\Omega$$
and call $d_\Omega$ the {\it order} of $\Omega$. 
The poles of the qudrature function $r_\Omega$ are called the {\it nodes} of $\Omega$, and we define
$$ n_\Omega=\#\{\text{distinct nodes}\}. $$

It can be shown (see e.g. \cite{1stpaper}) that the reflection in the unit circle, i.e. $z\mapsto 1/\overline z$, provides a one-to-one correspondence between the class of BQDs $\Omega$ of order $d + 1$ satisfying $0\in\Omega$ and the class of UQDs of order $d$ satisfying $0\notin \clos \Omega$.

Quadrature domains play an important role in several areas of analysis, see e.g. \cite{Gus05}. We refer to \cite{Sakai-book-QD} for a more
detailed discussion of basic properties and examples of QDs; see also \cite{1stpaper}.  Here we only mention that
the boundary of a quadrature domain is locally a real analytic curve except for finitely many cusps and double points; in fact the boundary is a real algebraic curve (up to a finite set of isolated points).

\subsection{Connectivity bounds}

Given a domain $\Omega\subset\widehat\CC$, we denote by $\conn\Omega$ the {\em connectivity of $\Omega$} so 
\begin{equation*}
	\conn\Omega=\#\{\text{components in $\Omega^c$}\}.
\end{equation*}
In \cite{1stpaper} we established the following connectivity bounds for quadrature domains.
\bigskip
\begin{customthm}{\!\!}
 If $\Omega$ is an unbounded quadrature domain of order $d_\Omega\geq 2$, then
\begin{equation}\label{eq-thma1}
\conn\Omega~\leq~ \min\{d_\Omega+n_{\Omega}-1,~ 2d_\Omega-2\}.
\end{equation}
If, in addition, there is a node at $\infty$, then
\begin{equation}\label{eq-thma2}
\conn\Omega~\leq~ d_\Omega+n_{\Omega}-2.
\end{equation}

If $\Omega$ is a bounded quadrature domain of order  $d_\Omega\ge3$, then
\begin{equation}\label{eq-thma3}
\conn\Omega~\leq~ \min(d_\Omega+n_{\Omega}-2,~2d_\Omega-4).
\end{equation}
If, in addition, there are no nodes of  multiplicity $\ge3$, then
\begin{equation}\label{eq-thma4}
\conn\Omega~\leq~ \min(d_\Omega+n_{\Omega}-3, ~2d_\Omega-4).
\end{equation}
\end{customthm}

Our goal is to show that these inequalities are best possible.

\subsection{Sharpness}

We will prove the sharpness of the connectivity bounds \eqref{eq-thma1}--\eqref{eq-thma4} for all integers $n$ and $d \le n$. Moreover, we will show that we can even prescribe
node multiplicities. The following theorems are the main results of the paper.
\bigskip
\begin{customthm}{A}\label{thm:UQD} Given positive integers $d\ge2$ and $n\le d$, and given a partition  $d=\mu_1+\dots+\mu_n$, there exists an unbounded quadrature domain $\Omega$ of order $d$ with $n$ finite nodes of multiplicities  $\mu_1,\dots,\mu_n$ such that
$$\conn(\Omega)=\min(d+n-1, ~2d-2).$$
Also, there  exists an unbounded quadrature domain $\Omega$ of order $d$ such that the infinity is a node of multiplicity $\mu_n$, the finite nodes have multiplcities $\mu_1,\dots, \mu_{n-1}$, and 
$$\conn(\Omega)=d+n-2.$$
\end{customthm}

\begin{customthm}{B}\label{thm:BQD}  Given positive integers $d\ge3$ and $n\le d$, and  a partition  $d=\mu_1+\dots+\mu_n$,
there exists a bounded quadrature domain $\Omega$ of order $d$ with $n$ nodes of multiplicities  $\mu_1,\dots,\mu_n$ such that
$$\conn(\Omega)=\min(d+n-3, ~2d-4).$$
If at least one of the $\mu_j$s is $\ge3$,  there exists an $\Omega$ such that
$$\conn(\Omega)=d+n-2.$$
\end{customthm}

The case $d = n$ has been known. Referring to Sakai's work, Gustafsson proved in \cite{Gu1} the existence of BQDs of connectivity $2d - 4$ for all $d > 3$.  
His argument  works for UQDs as well. 
Our proofs of Theorems \ref{thm:UQD} and \ref{thm:BQD} will use ideas of Gustafsson and Sakai (along with several other tools).

\subsection{Fixed points of anti-holomorphic rational maps}\label{sec:fixed}
The sharpness of connectivity bounds for {\em unbounded} quadrature domains (Theorem \ref{thm:UQD}) has the following application to the well-studied
problem concerning the number of (distinct) solutions to the equation
$$r(z) = \bar z,$$
where $r$ is a rational function. We will denote the number of distinct roots (including $\infty$) of the above equation by $\widehat F_r$.

A rather straightforward generalization of Khavinson-\'Swi\c atek \cite{Kha1} and Khavinson-Neumann \cite{Kha2} is the following statement.
\bigskip
\begin{customthm}{\!\!} Let $r$ be a rational function of degree $d\ge2 $ with  $n$  distinct poles.  
$$\widehat F_r\le
\min\{3d+2n-3, 5d-5\}.$$
\end{customthm}

We will show that this estimate is best possible in the following sense.
\bigskip
\begin{customthm}{C}\label{thm:C}
Given $d\geq2$, $n\leq d$, and given a partition $d=\mu_1+\cdots+\mu_n$, there exists a rational function $r$ with pole multiplcities $\mu_j\text{'s}$ such that
$$\widehat F_r= \min\{3d+2n-3, 5d-5\}.$$
Moreover, we can choose $r$ so that $r(\infty)\ne\infty$, and if $\mu_n>1$, then we can also choose $r$ so that $\infty$ is a pole of multiplicity $\mu_n$
\end{customthm}

Two special cases of Theorem \ref{thm:C} has already been known. Rhie \cite{Rh03} showed that for all $d$ there exists a rational function of degree $d$ with $5d-5$ finite fixed points (the case $n=d$), and Geyer \cite{Ge03} proved that there are polynomials with $3d-2$ finite fixed points
(the case $n=1$).


\subsection{Gravitational lensing, Crofoot-Sarason conjecture}\label{sec:GL-intro}

The theorems of Rhie and Geyer are actually stronger than the statements just concerning the existence of functions with a certain number of fixed points.

(a) In \cite{Rh03} Rhie constructed her examples in the class of functions of the form
$$r(z)=-\gamma\,z+s+\sum_{j=1}^N\frac{\varepsilon_j}{z-z_j},$$
where $s\in\CC$ and $z_j\in\CC$ but
\begin{equation}\label{eq:phys-constraint}   \varepsilon_j>0,~~~ \gamma\ge0.
\end{equation}
In gravitational lensing, the equation $r(z)=\bar z$ then has the following interpretation: $N$ is the number of gravitational masses $\varepsilon_j$ at positions $z_j$ in the lens plane, $s$ is the position of a light source, and $\gamma$ is the ``{\it shear}''. The finite fixed points of $\bar r$ are the images of the light source, and  Rhie showed that if $\gamma=0$ (so $\deg r=N$) then the number of images can be $5N-5$.

In Section 6 we will extend Rhie's result to the case of non-zero shear. We will show that 
{\it if $\gamma\ne0$ (so $\deg r=N+1$), then}  
the number of gravitational lensing images can be as large as $5N-1$, see Proposition \ref{thm:smallgamma} and Proposition \ref{cor:gravi}. Note that this fact does not follow from Theorem \ref{thm:C} because we need to satisfy the physical constraints \eqref{eq:phys-constraint}.

(b) In \cite{Ge03}, Geyer established the existence of polynomials with $3d-2$ fixed points by proving the Crofoot-Sarason conjecture which claimed that for every $d>1$  there exists a degree $d$ polynomial $p$ and distinct points $c_1,\cdots,c_{d-1}$ in $\CC$ such that $\overline{p(c_j)}=c_j$ and $p'(c_j)=0$.  Explicit examples were known for $d=2,3$ (Crofoot-Sarason, see \cite{Kha1}) and $d=4,5,6,8$ in Bshouty and Lyzzaik \cite{BL}. 

Geyer's argument was based on Thurston's classification of critically finite topological rational maps and Berstein-Levy's theorem on the absence of obstructing multi-curves for certain types of critically finite topological polynomials.
It is not clear how to extend Geyer's argument to the case of general rational functions. 

In Section \ref{sec:Crofoot-Sarason} we will give an alternative proof of Crofoot-Sarason conjecture. We will also prove analogous statements for rational function.

\subsection{Extreme quadrature domains and inscribed cardioid theorem}

Our argument in the proof of sharpness results is completely different from the arguments of Rhie and Geyer.  It works for all kinds of pole multiplicity data, and applies to  both  unbounded  and bounded quadrature domains.   Let us briefly describe the main idea.

A simply connected UQD of order $d$ is {\it extreme} if $\infty$ is the only node and the boundary has $d+1$ cusps and $d-2$ double points; those are the maximally possible numbers for a given degree. Similarly, a simply connected BDQ of order $d$ is {\it extreme} if the origin is the only node and the boundary has $d-1$ cusps and $d-2$ double points.  

\bigskip
\begin{customthm}{D}\label{thm:D}
Extreme quadrature domains exist in all orders. 
\end{customthm}

This fact is stated (somewhat implicitely) in Suffridge's paper \cite{Suffridge72}, see Section \ref{app:Suffridge}. Since we have had some difficulties in understanding the proof in \cite{Suffridge72}, we give a self-contained, more direct proof of Theorem \ref{thm:D} in Section \ref{sec:extreme}. (Our proof is still based on Suffridge's ideas.)

Extreme QDs have some perculiar geometry. If $\Omega$ is an extreme BQD, then the unbounded component of $\inte(\Omega^c)$ is {\it cardioid-like} but all bounded components are {\it deltoid-like}, see the definitions of {\it cardioid-like} and {\it deltoid-like} in Section \ref{sec:inscribed-card}. If $\Omega$ is an extreme UQD then all components are deltoid-like. We use extreme QDs as ``building blocks'' for the construction of QDs with particular connectivity properties. The  construction is based on the following statement.

{\em {\bf Inscribed cardioid Theorem.} One can inscribe any cardioid-like curve inside any deltoid-like curve so that there are at least 4 intersection points. }

\begin{figure}
\begin{center}
\includegraphics[width=0.2\textwidth]{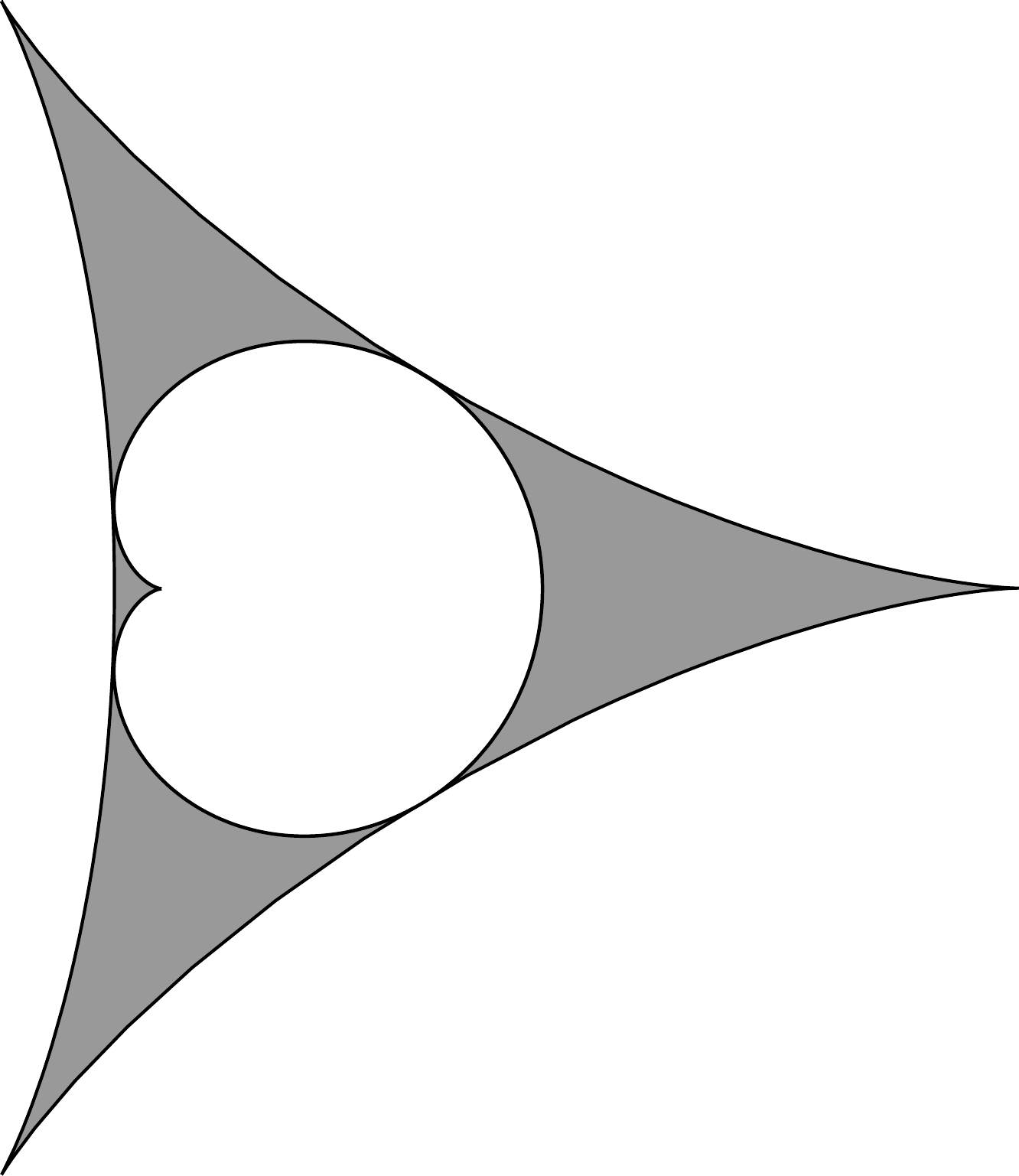}
\caption{\label{fig:card-in-deltoid} Cardioid inscribed in deltoid}
\end{center}
\end{figure}

See Figure \ref{fig:card-in-deltoid}  (with genuine cardioid and deltoid) for illustration.

\subsection{Organization of the paper} 

In Section \ref{sec:2} we study the geometry of extreme QDs and  prove their existence (Theorem \ref{thm:D}) in Section \ref{sec:extreme}. The proof of the inscribed cardioid theorem is in Section \ref{sec:packing}. In Section \ref{sec:main-proof} we prove Theorem \ref{thm:UQD} and \ref{thm:BQD}. 
The main tools here (in addition to extreme quadrature domains and the inscribed cardioid theorem) are Hele-Shaw deformations and quasi-conformal surgery from \cite{1stpaper}. 
Finally, in Section \ref{sec:rz} we prove Theorem \ref{thm:C} and discuss gravitational lensing and rational analogues of Crofoot-Sarason polynomials.

The authors would like to thank Dmitry Khavinson and Donald Sarason for valuable communication and references.

\section{Extreme quadrature domains}\label{sec:2}

\subsection{Suffridge polynomials in \texorpdfstring{$(S_d)$}{Sd}} 
We will denote by $(S_d)$ the set of polynomials
$$f(z)=z +a_2 z^2+\dots +a_d z^d$$
which are univalent in the unit disk $\mathbb D$, i.e.
 $$z_1\ne z_2\quad \Longrightarrow\quad f(z_1)\ne f(z_2).$$

\begin{defn}  Let $\Omega\subset\widehat \CC$ be a domain with compact boundary, and $\inte \clos\Omega=\Omega$.  A function $S:\clos\Omega\to\CC$ is called a {\em Schwarz function of $\Omega$} if $S$ is continuous, meromorphic in $\Omega$, and $S(z)=\overline z$ for all $z\in\partial \Omega$. 
\end{defn}
The following is well-known (e.g. p.154 in \cite{Davis}, or Lemma 3.1 in \cite{1stpaper}).
\bigskip
\begin{lemma}\label{Schwarz-QD} Let $\Omega$ be as above. Then $\Omega$ is a quadrature domain if and only if there exists a Schwarz function of $\Omega$.  In this case, we have
$$S(z)=r_\Omega(z)+\frac{1}{\pi}\int_{\Omega^c}\frac{\dd A(w)}{z-w},$$
where the integral is in the principal value sense.
\end{lemma}
\bigskip
\begin{lemma}\label{lem:suff-BQD}
If $f\in(S_d)$, then $\Omega=f(\DD)$ is a BQD of order d with a single node at the origin. The boundary of $\Omega$ has at most $d-1$ cusps and at most $d-2$ double points.  
\end{lemma}

The cusps are the points $f(\zeta)$ such that $|\zeta|=1$ and $f'(\zeta)=0$. The double points are the points $f(\zeta_1)=f(\zeta_2)$ with $|\zeta_1|=|\zeta_2|=1$ and  $\zeta_1\neq\zeta_2$.

The first part of the lemma has a converse: if $\Omega$ is a simply connected BQD of order $d$ with a single node at the origin then $\Omega=f(\DD)$ for some univalent polynomial of degree $d$.

\begin{proof} It is easy to see that the function
\begin{equation*}
	S(z)=\overline{f\big(1\big/\,\overline{f^{-1}(z)}\,\big)},~~~ ( z\in\DD),\qquad S(z)=\bar z,~~~  (z\in\TT),
\end{equation*}
is a Schwarz function of $\Omega$.
Therefore, $\Omega$ is a BQD.  
It is also clear that $0$ is the only pole of $S$, and the multiplicity of the pole is $d$.
By Lemma \ref{Schwarz-QD}, the same is true for the quadrature function $r_\Omega(z)$.

The cusps correspond to the zeros of $f'$ on the unit circle so there can be at most $d-1$ cusps. 

The connectivity bound \eqref{eq-thma3} implies that a BQD of order $d$ with a single node can have at most $d-1$ connectivity components in its complement. A simple argument with Hele-Shaw flow, see \cite{1stpaper}, allows us to make a stronger statement -- no more than $d-1$ components in the interior of the complement. It follows that a bounded quadrature domain with a single node of order $d$ can have at most $d-1$ components in the complement of its closure.
Therefore we have
$$1+\#\{\text{double points}\}\leq d-1. $$
\vspace{-1.5cm}

\end{proof}
\bigskip
\begin{defn}
$f\in (S_d)$ is a {\em Suffridge polynomial} if the quadrature domain $f(\DD)$ is {\it extreme}, i.e. it has the maximal number of singularities: $d-1$ cusps and $d-2$ double points.   The curve $f(\TT)$ is a {\em Suffridge curve}.  
\end{defn}

Examples of Suffridge curves (see Figure \ref{fig:BQD} and \ref{fig:suff2}):

\begin{description}
\item[(a)] $d=2$: The only (up to rotation) Suffridge polynomial in $(S_2)$ is
$$f(z)=z+\frac{z^2}2.$$
The corresponding curve is the cardioid.
\item[(b)] $d=3$ (K\"ossler `51 \cite{Kossler}; Cowling, Royster '67 \cite{Cowling-Royster}; Brannan '67 \cite{Brannan}): 
$$f(z)=z+\frac{2\sqrt 2}3z^2+\frac{z^3}3.$$
This polynomial simultaneously maximizes $|a_2|$ and $|a_3|$ in $(S_3)$.
\item[(c)] $d=4$ (C. Michel '70 \cite{Michel})
$$ f(z)=z+\frac{3}{2} A\, \ee^{\ii t} z^2+A\, \ee^{-\ii t} z^3+\frac{z^4}{4},
$$
where $A=\frac{1}{2} \sqrt{3 \left(\sqrt{15}-3\right)}$ and $t=\frac{1}{3} \cos ^{-1}\left(\frac{3}{16} \sqrt{9+5 \sqrt{15}}\right)$.  
\item[(d)] $d=5$: This polynomial corresponds to the third picture in Figure \ref{fig:BQD}.
\begin{equation*}
f(z)=\frac{z^5}{5}+\frac{4}{5} \sqrt{\frac{2}{3}} z^4+\frac{6 z^3}{5}+\frac{8}{5} \sqrt{\frac{2}{3}} z^2+z.
\end{equation*}
\item[(e)] $d=6$ and $d=7$: The curves in Figure \ref{fig:suff2} were obtained numerically.  
\end{description}

\begin{figure}[ht]
\begin{center}
\includegraphics[width=3.5cm]{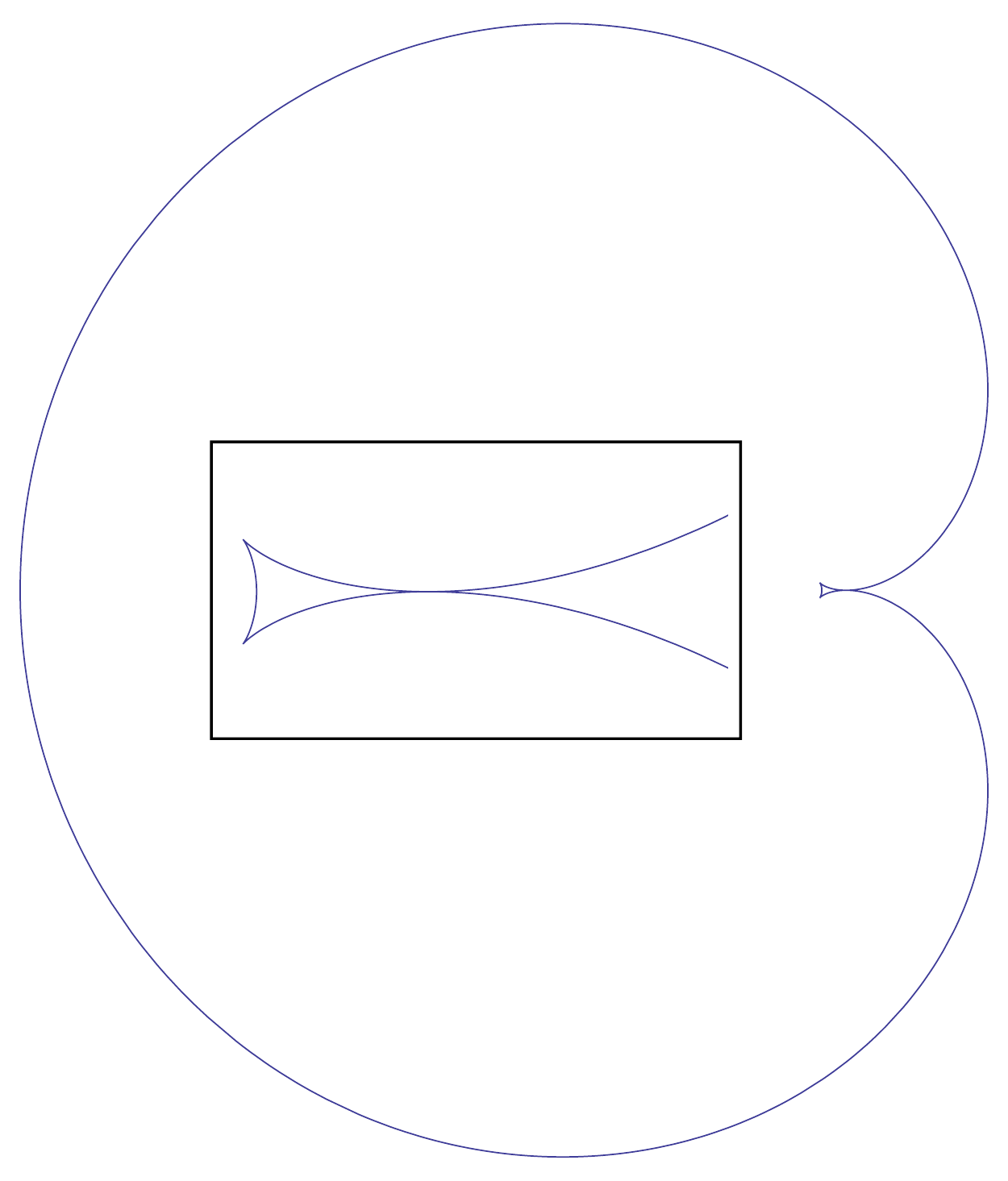}
~~\includegraphics[width=3.5cm]{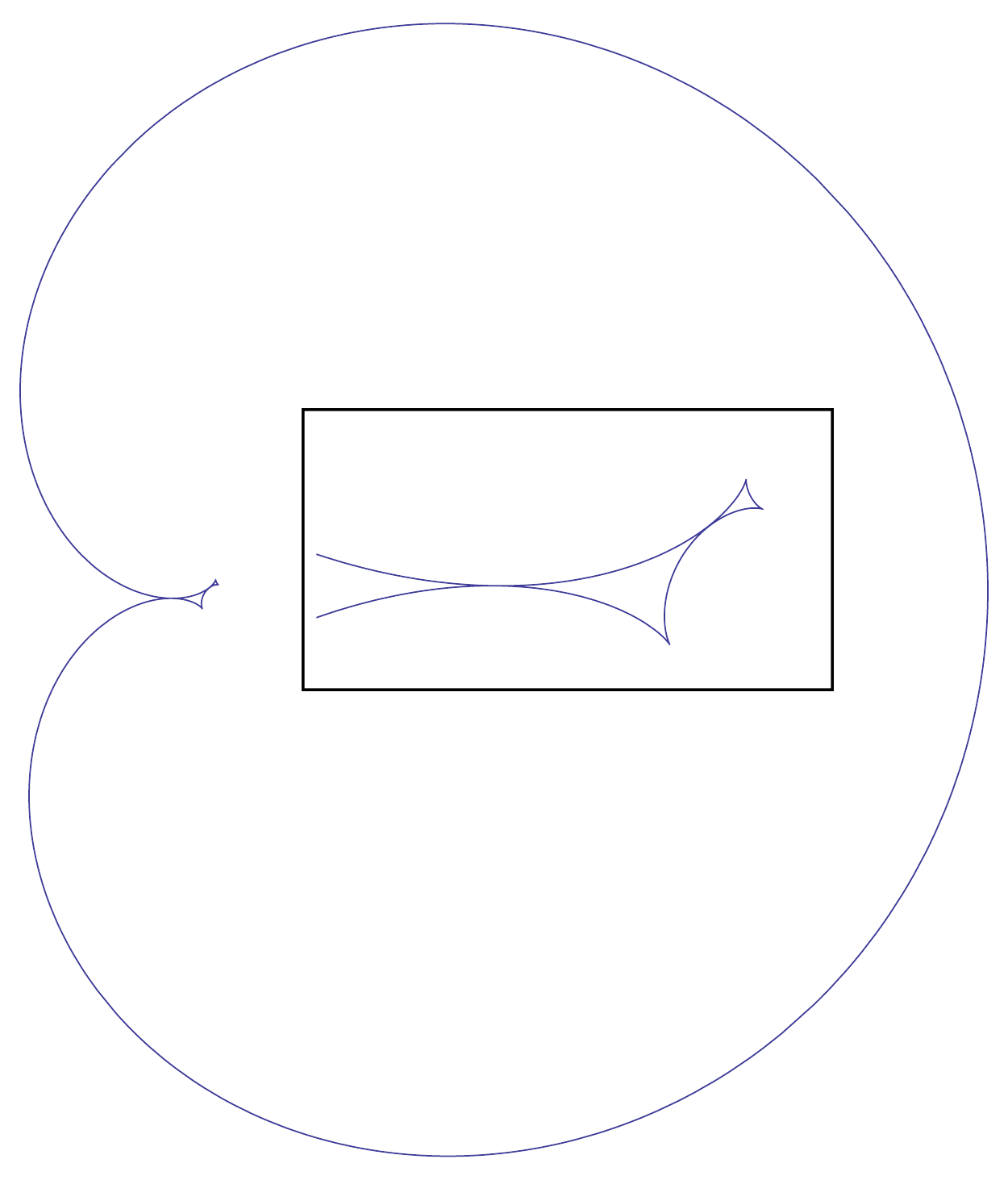}
~~\includegraphics[width=3.5cm]{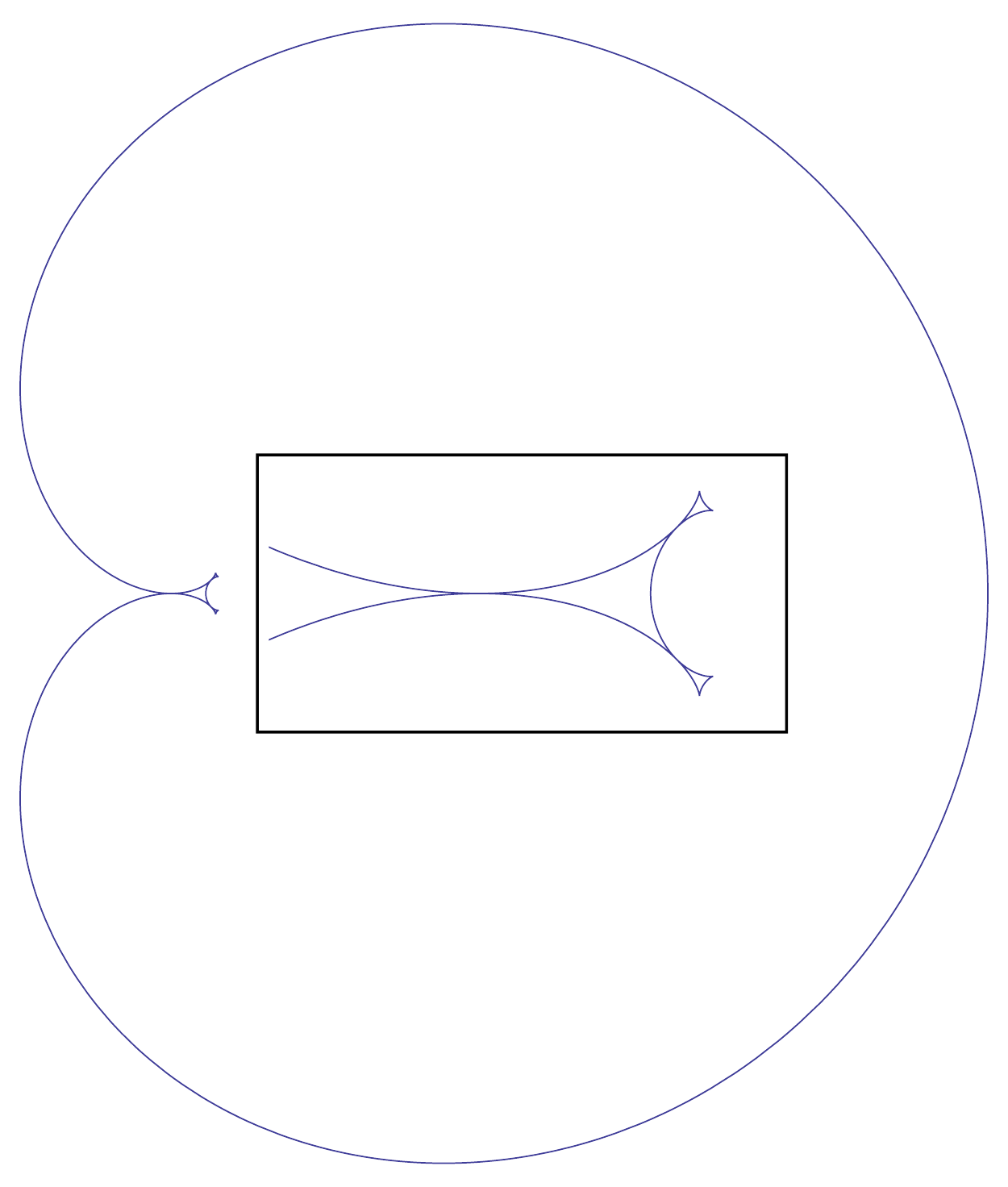}
~~\includegraphics[width=3.7cm]{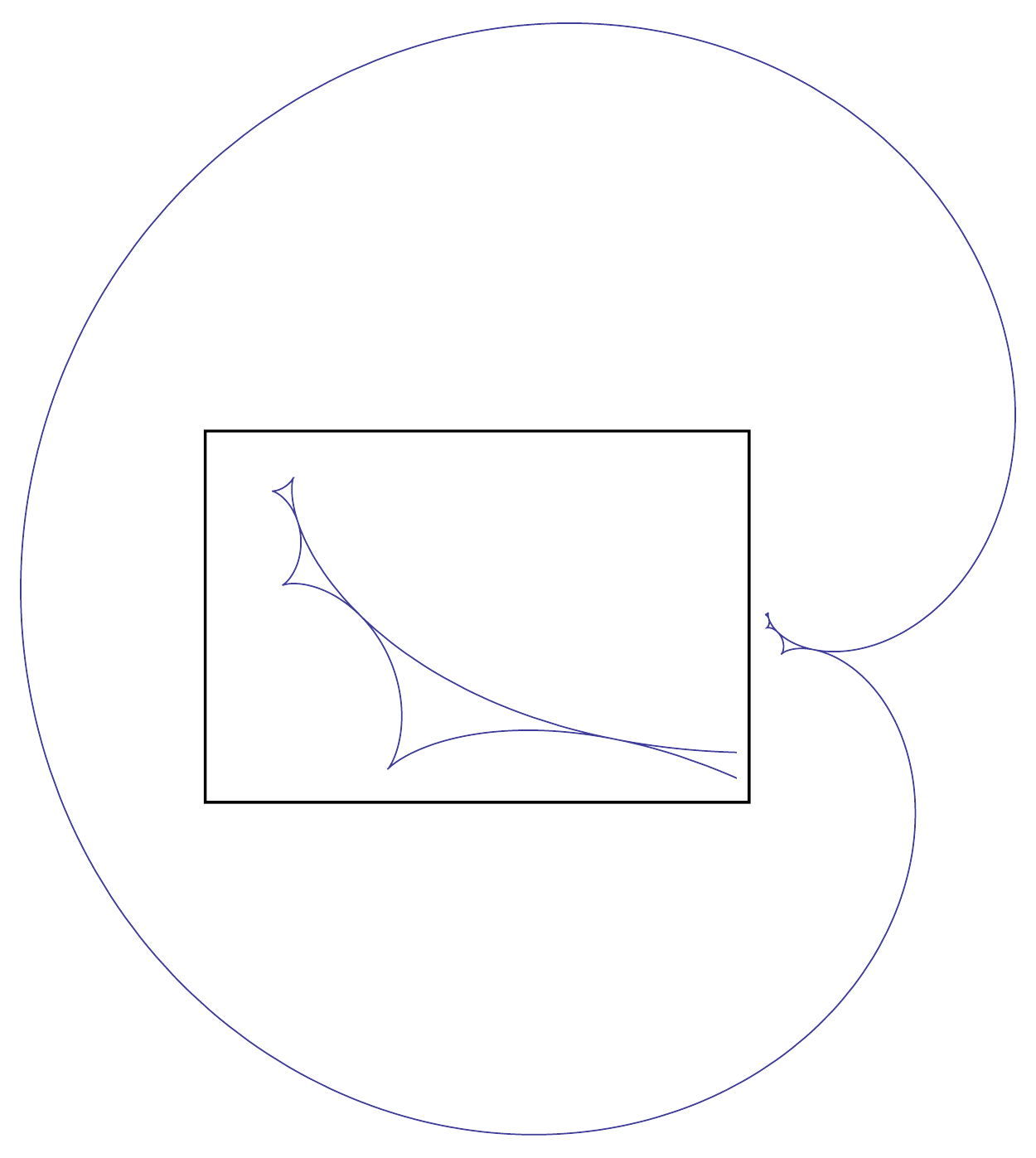}
\caption{\label{fig:BQD} Suffridge curves for $d=3$, $d=4$, $d=5$ and $d=5$ (from left to right). The boxed insets are zoomed images of the singularities.}
\end{center}
\end{figure}

\begin{figure}[ht]
\begin{center}
\includegraphics[width=5cm]{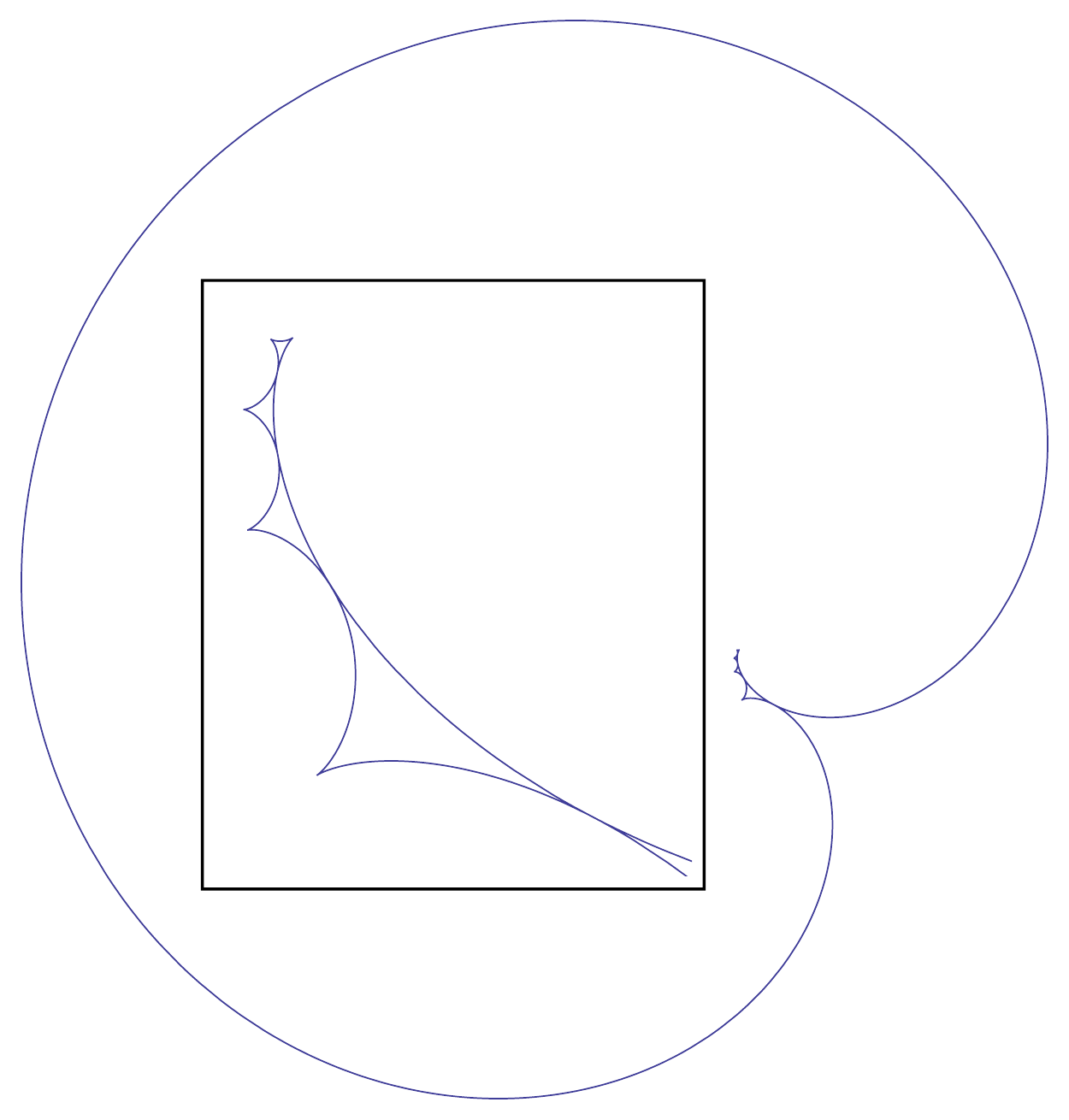}
~~~~\includegraphics[width=5cm]{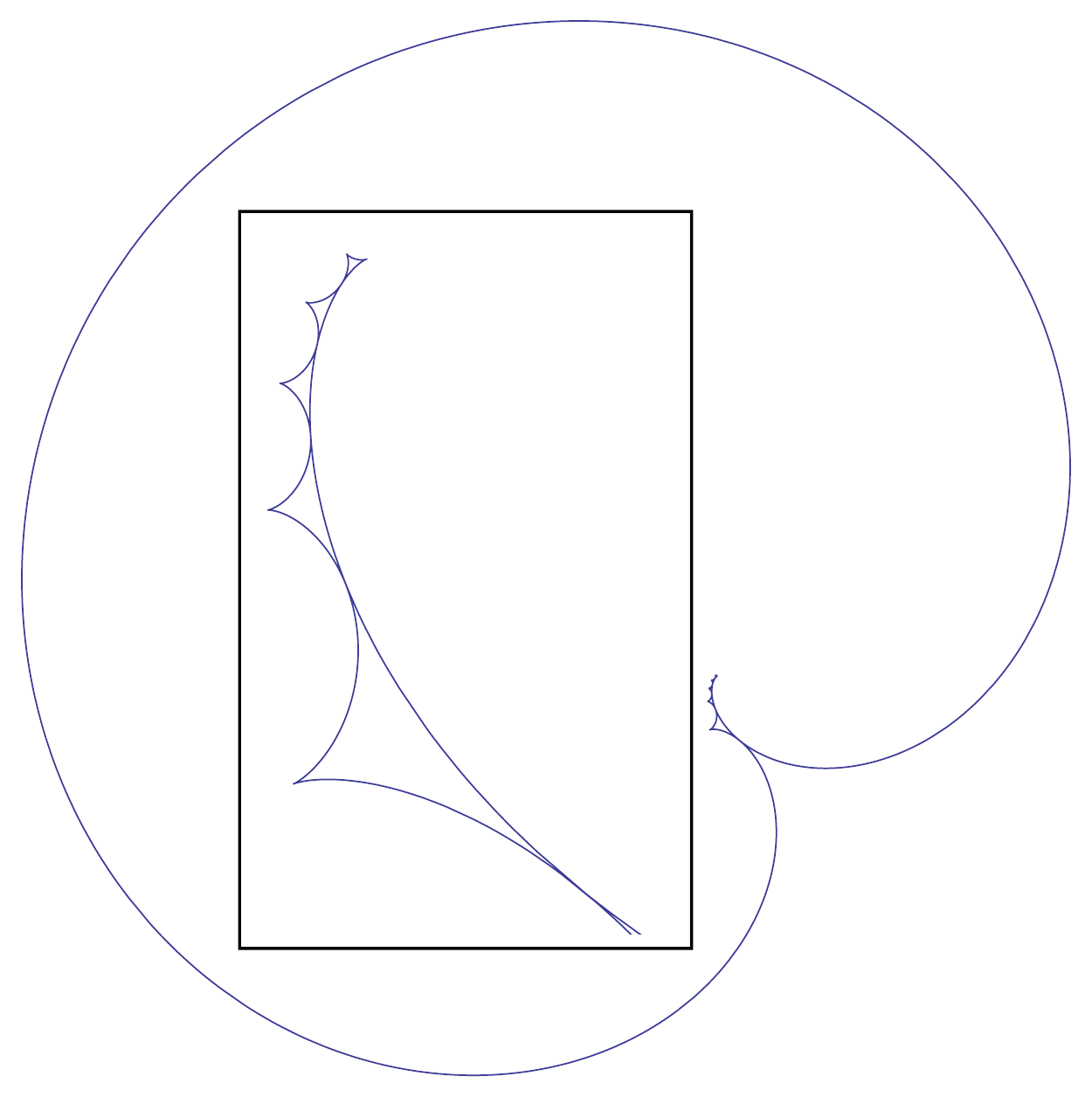}
\caption{Suffridge curves for $d=6$ and $d=7$. \label{fig:suff2}}
\end{center}
\end{figure}

Let us define
$$ (S^*_d)=\{f\in (S_d)\,|\,a_d=1/d\}$$
and consider $(S^*_d)$ as a compact subset of a finite dimensional Euclidean space (polynomials of degree $\le d$).
The existence of Suffridge polynomials in $(S_d)$ for all $d$ follows from Krein-Milman theorem on extreme points and the following fact (cf. \cite{Suffridge72}) which will be explained in the next section.  
\bigskip
\begin{thm}\label{thm:Sd} Extreme points of $(S^*_d)$ are Suffridge polynomials.\end{thm}

\subsection{Suffridge polynomials in \texorpdfstring{$(\Sigma_d)$}{Sigma}} 
We need an analogous theory for conformal maps in the exterior unit disc $\Delta$. We denote by $(\Sigma_d)$ the set of functions of the form
$$f(z)=z +\frac{a_1}{z}+\dots +\frac{a_d}{z^d}$$
which are univalent in $\Delta=\widehat\CC\setminus\clos\DD$.

\bigskip
\begin{lemma} If $f\in (\Sigma_d)$, then $\Omega=f(\Delta)$ is an UQD of order $d$  with a single node at $\infty$. The boundary of $\Omega$ has at most $d+1$ cusps and at most $d-2$ double points. \end{lemma}
\begin{proof}
The proof is exactly parallel to that of Lemma \ref{lem:suff-BQD}.
Using the relation
\begin{equation}
S(f(z))=\overline{f(1/\overline z)}
=\frac{1}{z} + \overline a_1  z +\dots +\overline{a_d}^d  z^d,
\end{equation}
we see that the only pole of the Schwarz function is at the infinity and the multiplicity of the pole is $d$. This means that $\Omega$ is a UQD of order $d$ with a single node at $\infty$.  The boundary can have at most $d+1$ cusps because $f'(z)$ can have at most $d+1$ zeroes.  The number of double points must be {\em at least one} less than the number of components in $(\clos\Omega)^c$, which is bounded by $d-1$ according to \eqref{eq-thma2}.     
\end{proof}

\begin{defn} $f\in (\Sigma_d)$ is a {\em Suffridge polynomial} if the quadrature domain $f(\DD)$ has  $d+1$ cusps and $d-2$ double points.   The curve $f(\TT)$ is a {\em Suffridge curve}.
\end{defn}

{\bf Example}.
\begin{description}
\item[(a)] $d=2$: The only (up to rotation) Suffridge polynomial in $(\Sigma_2)$ is
$$f(z)=z-\frac1{2z^2}.$$
The corresponding curve is the deltoid, see Figure \ref{fig:card-in-deltoid}.
\item[(b)] $d=3$: A Talbot's curve \cite{Talbot} is given by the Suffridge polynomial:
$$f(z)=z+\frac2{3z}-\frac1{3z^3}.$$
\item[(c)] $d=4$: 
$$f(z)=z-\frac5{8z}-\frac5{16z^2}-\frac1{4z^4}.$$
\item[(d)] $d=5$: The curve shown in Figure \ref{fig:UQD} is given by the Suffridge polynomial
$$f(z)=z+\frac{2\sqrt2}{5z^2}-\frac1{5z^5}.$$
\item[(e)] The curve for $d=7$ in Figure \ref{fig:UQD} and the two curves in \ref{fig:UQD-num} are obtained numerically. 
\end{description}
\bigskip

\begin{figure}
\begin{center}
\begin{tikzpicture}
\begin{scope}[yshift=-3.5cm,xshift=-0.2cm]
\node{\includegraphics[width=3.3cm]{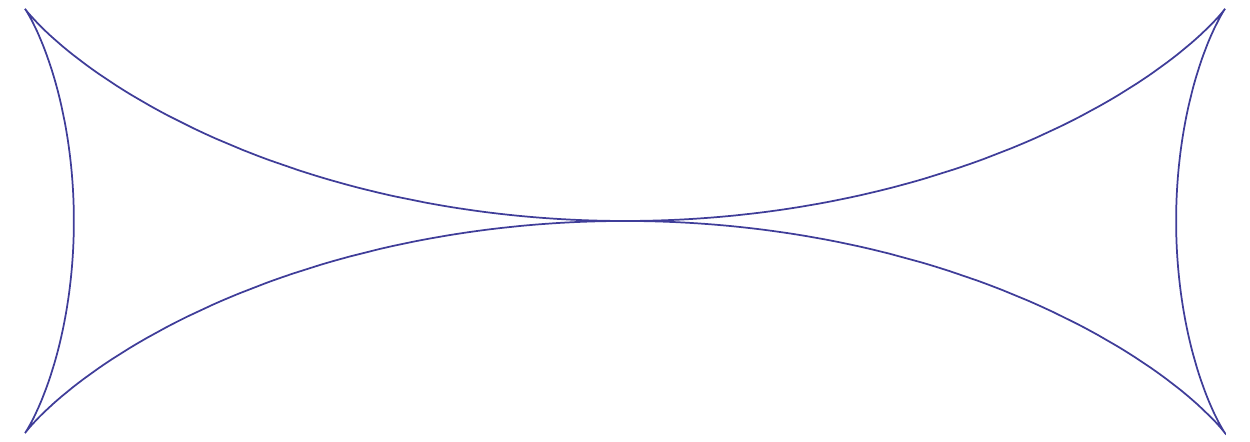}};
\end{scope}
\begin{scope}[yshift=-4.1cm,xshift=-0.2cm]
\node{$d=3$};
\end{scope}
\begin{scope}[yshift=-5.9cm]
\node{\includegraphics[width=4.5cm]{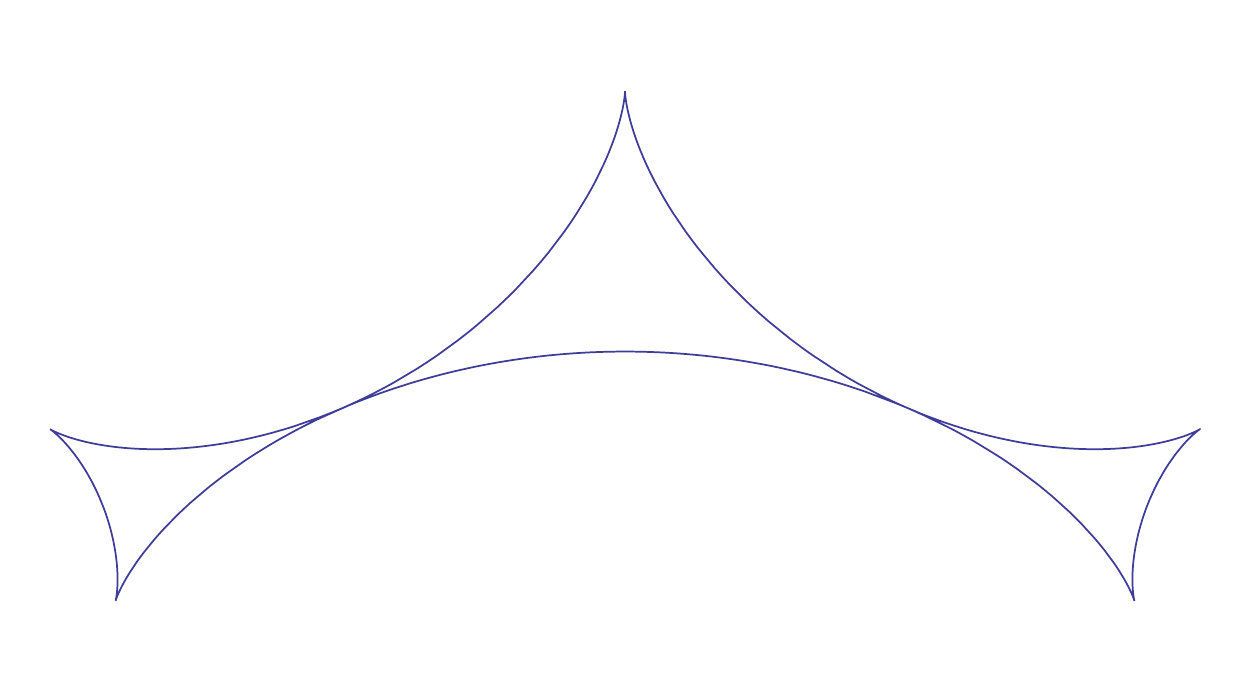}};
\end{scope}
\begin{scope}[yshift=-6.8cm]
\node{$d=4$};
\end{scope}
\begin{scope}[xshift=5.4cm,yshift=-4.5cm]
\node{\includegraphics[width=3.6cm]{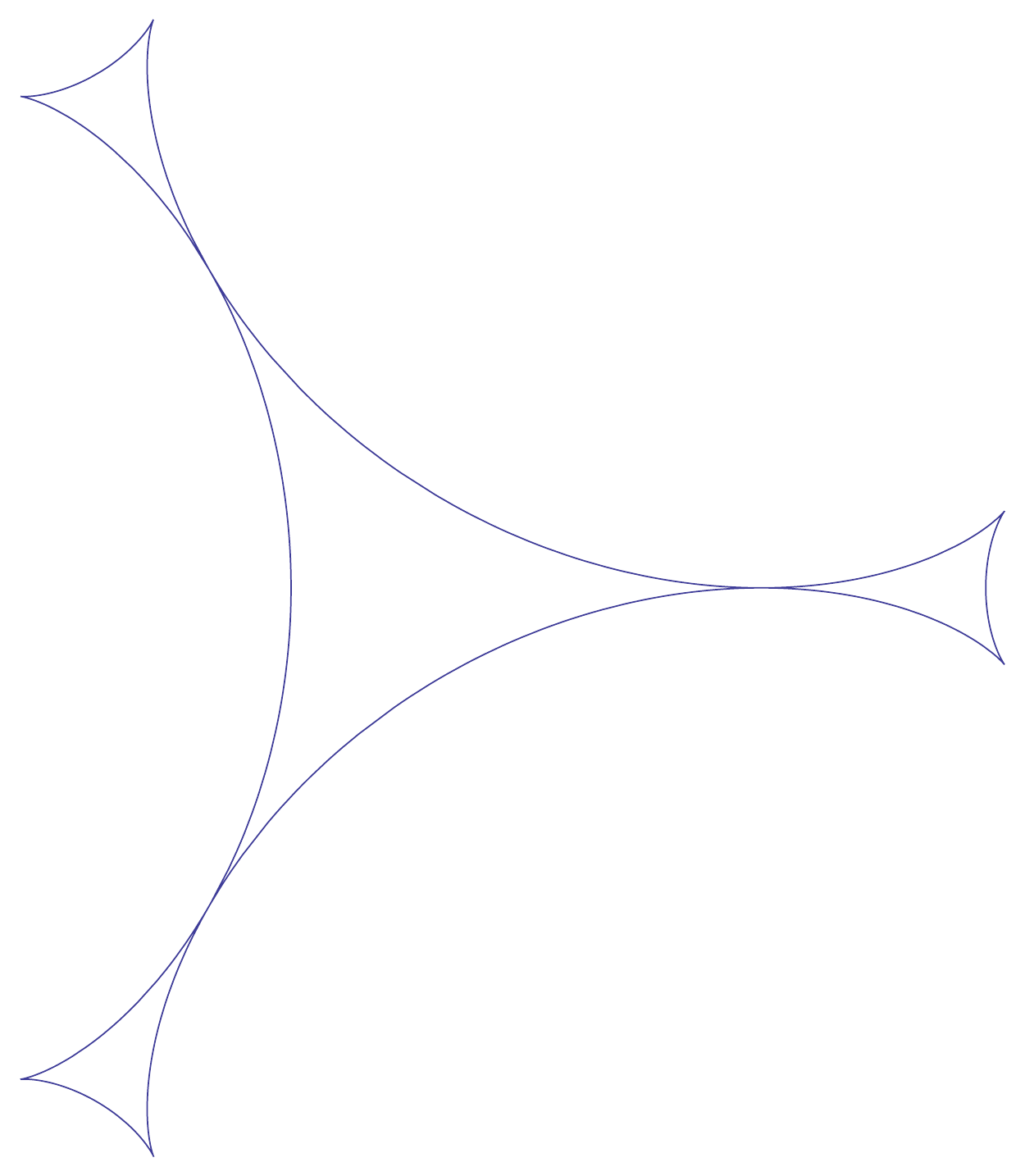}\!\!\!\!\!\!\!\!\!\!\!\!\!\!\!\!\!\!$d=5$};
\end{scope}
\begin{scope}[xshift=11cm,yshift=-4.5cm]
\node{\includegraphics[width=4.5cm]{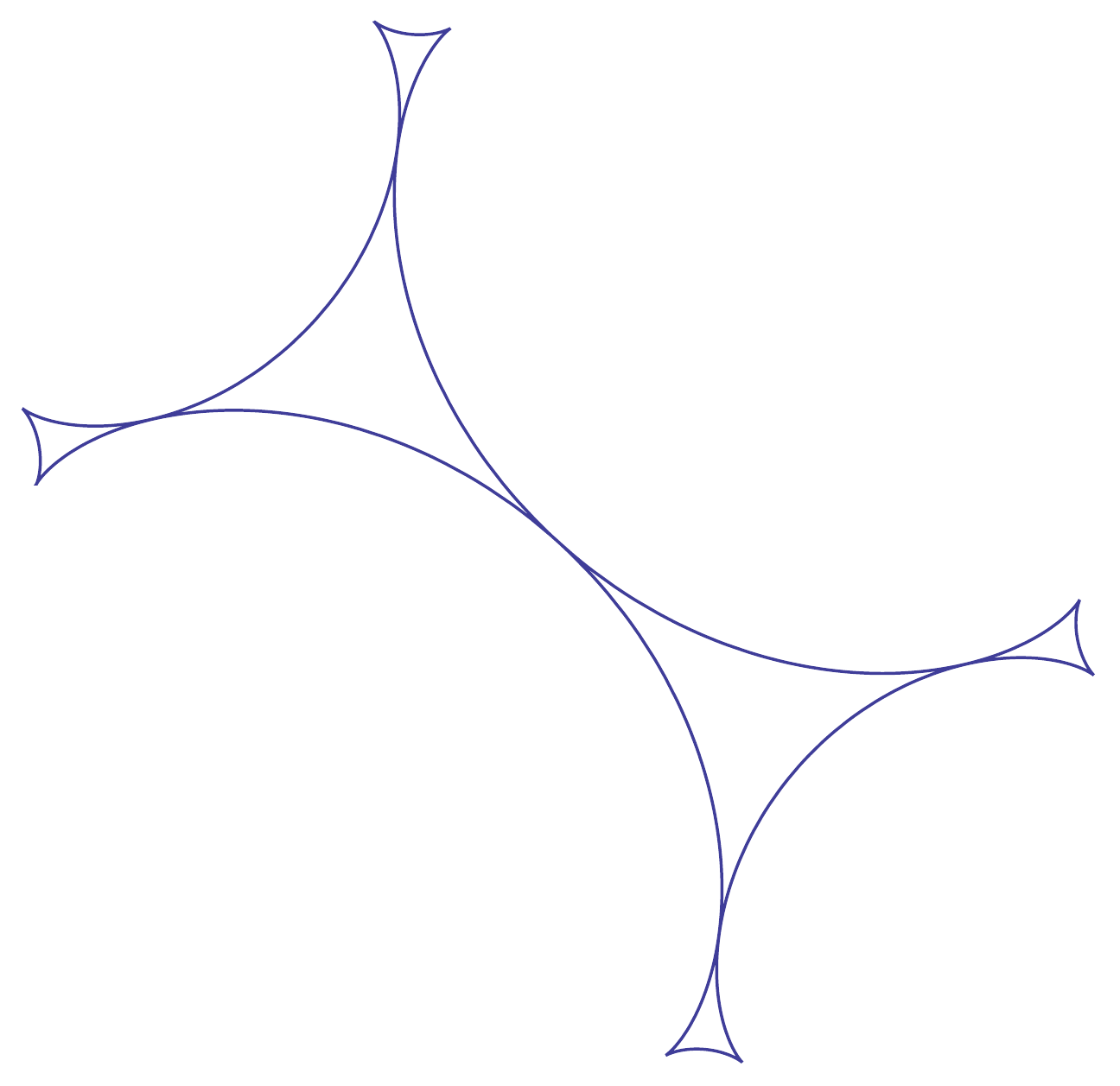}\!\!\!\!\!\!\!\!\!$d=7$};
\end{scope}
\end{tikzpicture}
\end{center}
\caption{\label{fig:UQD} Suffridge curves for $d=3,4,5$ and $7$}
\end{figure}

\begin{figure}
\begin{center}
\includegraphics[width=5.8cm]{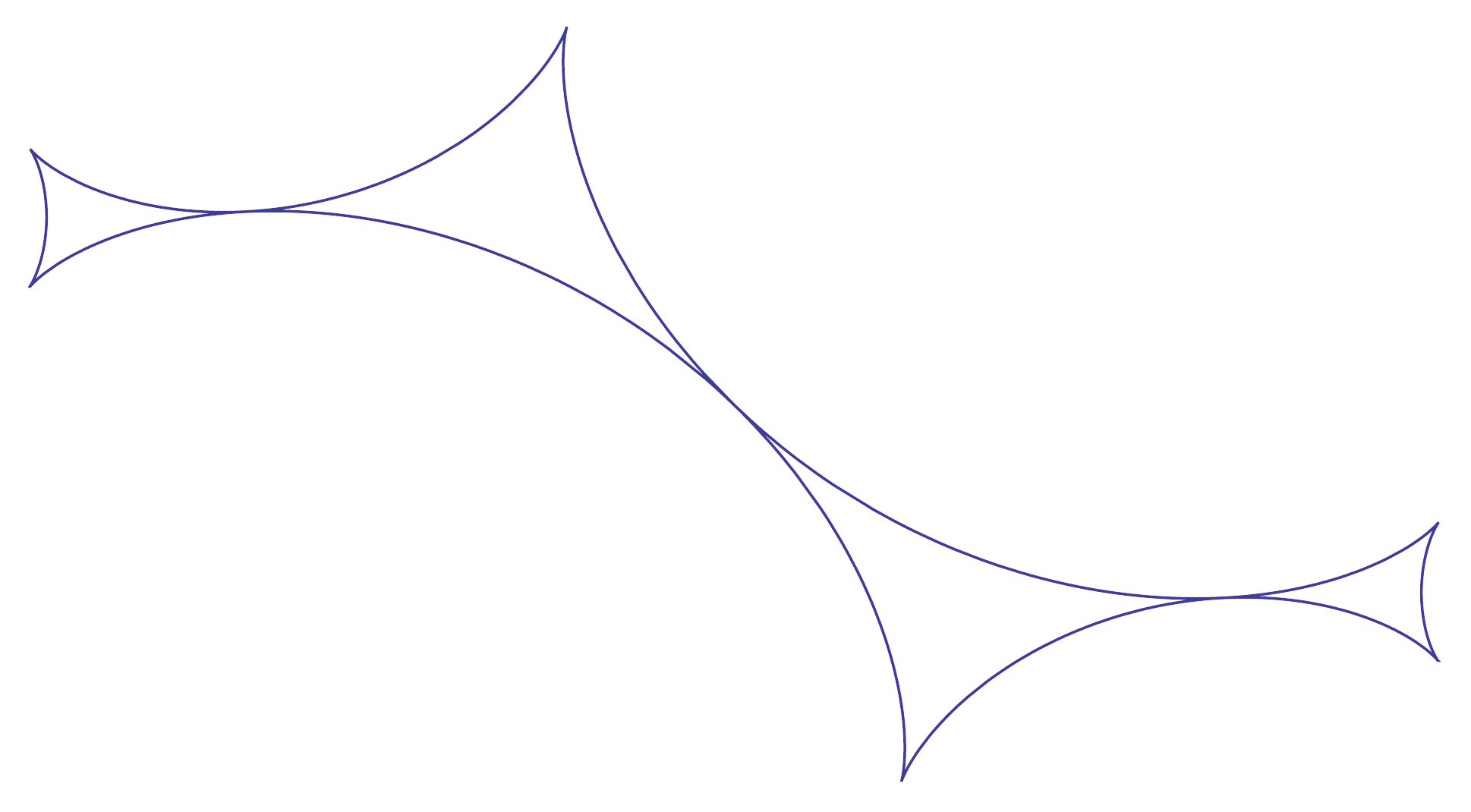}
\hspace{1cm}
\includegraphics[width=4cm]{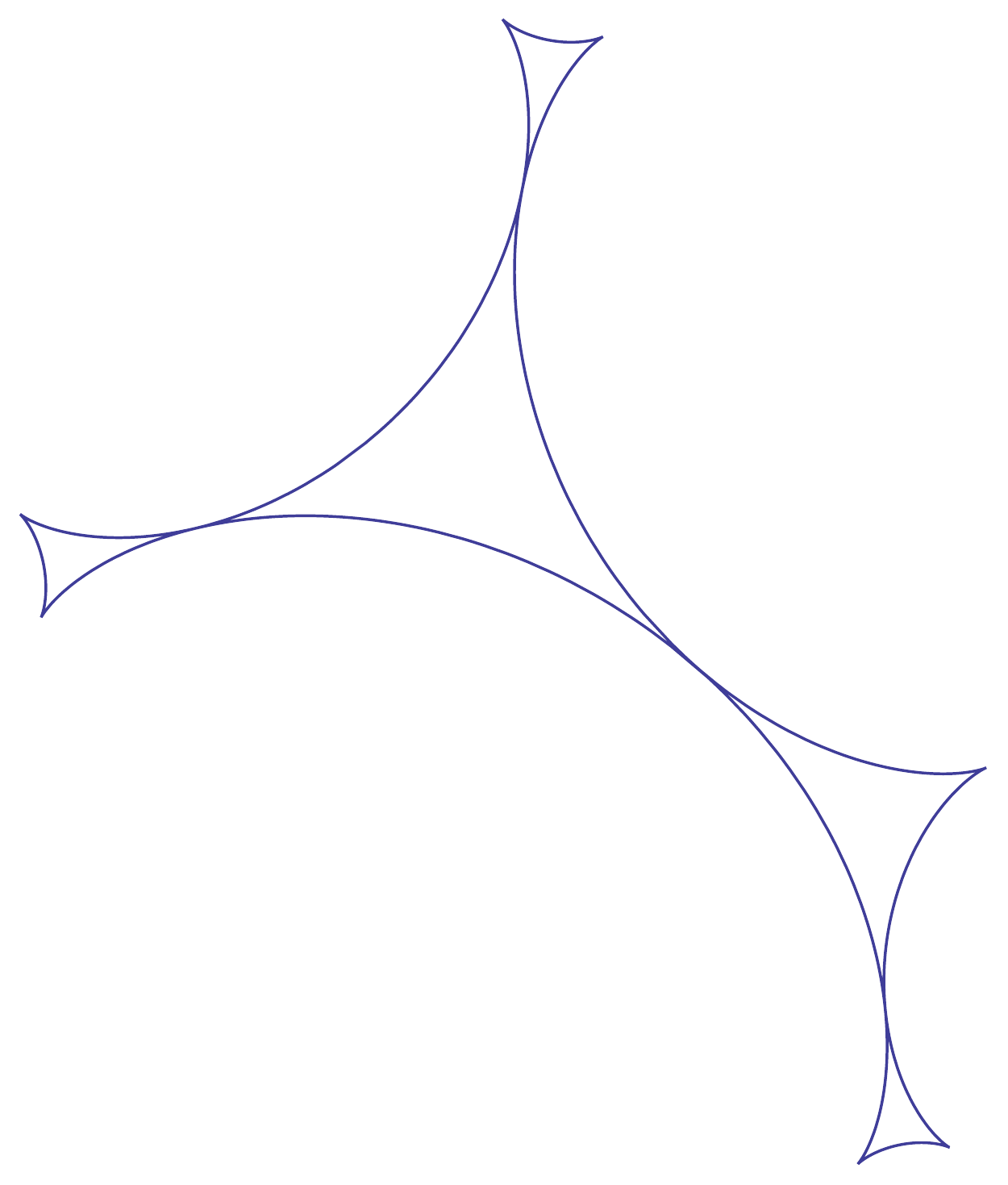}
\end{center}
\caption{\label{fig:UQD-num} Suffridge curves for $d=5$ and $d=6$}
\end{figure}

We will prove the following theorem in Section \ref{sec:extreme}. Denote
$$(\Sigma^*_d)=\{f\in (\Sigma_d):~ a_d=-1/d\}.$$

\begin{thm}\label{thm:Sigma} Extreme points of $(\Sigma^*_d)$ are Suffridge polynomials.\end{thm}

\subsection{Constant conformal curvature}

Let $\Pi_k$ be the space of polynomials of $\deg\leq k$.
An {\em involution in $\Pi_k$} is defined by
\begin{equation}\nonumber
p(z)\mapsto	p^*(z):=z^{k}\,\overline{p(1/\overline z)}.
\end{equation}
A polynomial $p$ is {\em self-dual in $\Pi_k$} if $p=p^*$.  

\bigskip
\begin{lemma}\label{lem:d1d1} (i) If $f\in (S^*_d)$ then $f'$ is self-dual in $\Pi_{d-1}$.
\\
(ii)
If $f\in (\Sigma^*_d)$ then $f'(1/z)$ is self-dual in $\Pi_{d+1}$.
\end{lemma}
\begin{proof}
(i) If $f\in (S_d^*)$, then the product of its $d-1$ critical points (zeros of $f'$) is $(-1)^{d-1}$ because $f'(z)=1+\cdots+z^{d-1}$.  This means that, if there is a critical point outside $\clos\DD$, there must be a critical point in $\DD$, which is impossible because $f$ is univalent in $\DD$.   Therefore, all the critical points must be on $\TT$.   Let $\{\zeta_j|j=1,\cdots,d-1\}$ be the critical points on $\TT$.  We have
$$  (f')^*(z)= z^{d-1}\overline{ (f')(1/\overline z)} = z^{d-1} \prod_{j=1}^{d-1}\overline{(1/\overline z- \zeta_j)}=\frac{(-1)^{d-1}}{\prod_{j=1}^{d-1}\zeta_j}\prod_{j=1}^{d-1}(z-\zeta_j)=f'(z), $$
where we used $1/\overline{\zeta_j}=\zeta_j$.
This implies that $f'$ is  {\em self-dual in $\Pi_{d-1}$}.
\\
(ii) If $f\in(\Sigma_d^*)$, then the product of its $d+1$ critical points is $(-1)^{d+1}$.  For $f$ to be univalent, all the critical points must be on $\TT$.  This makes the polynomial $f'(1/z)$ {\em self-dual in $\Pi_{d+1}$}.
\end{proof}
\begin{remark}
The proof also shows that the curve $f(\TT)$ has the maximal number of cusps, which correspond to the critical points of $f$. Note that all zeros of $f'$ are {\em simple} because $f$ is univalent.
\end{remark}
\bigskip
\begin{lemma}\label{lem:kappa}
(i) Let $f$ be a polynomial such that $f'$ be self-dual in $\Pi_k$.  Consider the curve $f(\TT)$ with parametrization $t\mapsto f(\ee^{it})$. The unit tangent vector $\tau$ of $f(\TT)$ at $f(\zeta)$, $\zeta\in\TT$, is given by the formula
$$\tau(\zeta)={\rm sgn}(A(\zeta))\, i\zeta^{(k+2)/2},\quad \zeta\in\TT\setminus\{\text{critical points of $f$}\},$$   
where $A:\TT\to\RR$ is a continuous function which changes its sign exactly at each critical point of $f$.
\\
(ii) Similarly, let $f$ be a rational function such that $f'(1/z)$ is a polynomial which is self-dual in $\Pi_k$. 
Then the unit tangent vector $\tau$ of $f(\TT)$ is given by $$\tau(\zeta)={\rm sgn} (A(\zeta))\, i\zeta^{(-k+2)/2},\quad \zeta\in\TT\setminus\{\text{critical points of $f$}\},$$
where $A:\TT\to\RR$ is a continuous function which changes its sign exactly at each critical point of $f$.
\end{lemma}
\begin{proof} For the case (i), the self-duality of $f'$ means that, for $\zeta\in\TT$,
\begin{equation*}
2i\,{\rm Im}\Big( \zeta^{-k/2}f'(\zeta)\Big)= \zeta^{-k/2}f'(\zeta) - \zeta^{k/2}\overline{f'(1/\overline\zeta)}=0.
\end{equation*}
This means that $f'(\zeta)=A(\zeta)\, \zeta^{k/2}$ for some continuous function $A:\TT\to\RR$.  Also, $A(\zeta)$ changes sign at each zero of $f'$ because the zeros are simple.
The tangent vector is given by 
\begin{equation*}
  \frac{\dd f(\ee^{i t})}{\dd t}=i \zeta f'(\zeta)\Big|_{\zeta=\ee^{i t}}=i A(\zeta) \,\zeta^{(k+2)/2},
\end{equation*}
and the {\em unit} tangent vector is as stated.

Similarly, for the case (ii), we have that $f'(1/\zeta)=A(\zeta)\,\zeta^{k/2}$ for some continuous function $A:\TT\to\RR$.  Then the unit tangent vector is given by
\begin{equation*}
\tau=\frac{1}{|A(\ee^{it})|}  \frac{\dd f(\ee^{i t})}{\dd t}=\frac{1}{|A(\zeta)|}i \zeta f'(\zeta)\Big|_{\zeta=\ee^{i t}}=i \frac{A(\zeta)}{|A(\zeta)|} \,\zeta^{(-k+2)/2}.
\end{equation*}
\vspace{-1.3cm}

\end{proof}

Recall that the {\em curvature} $k$ of the curve $t\mapsto f(\ee^{it})$ is given by the function
$$t\mapsto\frac{1}{|f'(\ee^{it})|}\frac{\dd\arg \tau(\ee^{it})}{\dd t}.$$
The curvature measures how fast the angle of the tangent vector changes as one moves along the curve with unit velocity.      The curvature is positive when the curve turns left of the tangent. 

In our case, it will be convenient to define {\it conformal curvature} at $\zeta=\ee^{it}\in \TT$ by
\begin{equation*}
\kappa(\zeta)=\frac{\dd \arg \tau(\ee^{i t})}{\dd t}.
\end{equation*}
This is the angular velocity of the tangent vector as one moves along the curve with the velocity $|f'|$. 
Note that $\kappa$ has the same sign as the usual curvature.

The last two lemmas (Lemma \ref{lem:d1d1} and Lemma \ref{lem:kappa}) have the following corollary, which says that Suffridge curves have {\em constant} conformal curvature (except for cusps).
\bigskip
\begin{cor}\label{cor:kappa} (i) If $f\in (S^*_d)$ then 
\begin{equation}\label{eq:kappa1}
\kappa(\zeta)=\frac{1+d}2,\qquad \zeta\in \TT\setminus\{\text{critical points of $f$}\}.\end{equation}
(ii)
If $f\in (\Sigma^*_d)$ then 
\begin{equation}\label{eq:kappa2}
  \kappa(\zeta)=\frac{1-d}2,\qquad \zeta\in \TT\setminus\{\text{critical points of $f$}\}.
\end{equation}
\end{cor}

\subsection{Geometry of Suffridge curves}

Since $\kappa$ is the rate of change for the angle of the tangent, and since the angle jumps by $\pi$ (i.e. the tangent vector changes the direction by $\pi$ or, equivalently, $A(\zeta)$ in Lemma \ref{lem:kappa} changes the sign) at the cusps, we have
$$\int_0^{2\pi}\kappa(\ee^{\ii t})\,\dd t + \pi \cdot \#\{\text{cusps}\}=0 \mod 2\pi.$$

Similarly, given a double point, $f(\ee^{\ii t^+})=f(\ee^{\ii t^-})$ for $t^-<t^+<t^-+2\pi$, we have
$$ \int_{t^-}^{t^+}\kappa(\ee^{\ii t})\,\dd t  + \pi\cdot \#\{\text{cusps in $(t^-,t^+)$}\} = \pi \mod 2\pi,$$
because the two tangent vectors at the double point have opposite directions.

Combining the last equation with Corrollary \ref{cor:kappa} we get the following useful relation for a double point:
\begin{equation}\label{eq:one}
 \ee^{\ii (d\pm 1)(t^+-t^-)}=1,\quad \text{$+$ for $(S^*_d)$ and $-$ for $(\Sigma^*_d)$.}
 \end{equation}

\bigskip
\begin{lemma}\label{lem:3} Let $\Gamma\subset\CC$ be a Jordan curve with counterclockwise orientation. Suppose $\Gamma$ is smooth except for $N$ cusps and has negative curvature at all regular points. Then $N\ge 3$. \end{lemma}
\begin{proof} 
The Jordan domain bounded by $\Gamma$ has zero angles at the cusp points. Let us approximate $\Gamma$ by a smooth Jordan curve $\Gamma_\epsilon$ so that the curves coincide outside the $\epsilon$-neighborhoods of the cusps.
Note that as $\epsilon$ goes to zerp, the angle of the tangent of $\Gamma_\epsilon$ changes by $+\pi$ across the $\epsilon$-neighborhood around each cusp (note the counterclockwise orientation on the curve).  Since the total variation of the angle of the tangent $\Gamma_\epsilon$ is $2\pi$ and, for the curve $\Gamma$ we get (in the limit)
\begin{equation}\label{eq:fullrotate}
\int_\Gamma k\, \dd s+\pi \#\{\text{cusps}\}=2\pi,
\end{equation}
where $k$ is the usual curvature and $\dd s$ is the arclength.  parametrization in counterclockwise direction.
The lemma follows because the integral in \eqref{eq:fullrotate} is negative. .\end{proof}

\begin{thm}\label{thm:geom2} Let $f$ be a Suffridge polynomial in $(\Sigma_d)$ and let  $U_j$'s, ($1\le j\le d-1$), be the complementary components of ${\CC}\setminus f(\overline {\Delta})$. Then the boundary of each domain $U_j$ has exactly three singular points.\end{thm}
\begin{proof}  Let $d_j$ and $c_j$ be the number of double points and cusp points in $\partial U_j$.   By the definition of Suffridge polynomials in $(\Sigma_d)$ we have
$$\sum d_j=2d-4,\qquad \sum c_j=d+1.$$
Adding up, we have 
$$\sum_{j=1}^{d-1} (d_j+c_j)=3d-3,\quad \text{i.e.} \quad \sum_{j=1}^{d-1} (d_j+c_j-3)=0. $$
By Lemma \ref{lem:3} we have $d_j+c_j-3\geq0$ and therefore we have $d_j+c_j-3=0$.\end{proof}

\begin{thm}\label{thm:geom1} Let $f$ be a Suffridge polynomial in $(S_d)$ and let $U_\infty$ and $U_j$'s ($1\le j\le d-2$) be the unbounded and bounded complementary components of $\CC\setminus f(\overline \DD)$. Then $\partial U_\infty$ has a single singular point (a cusp if $d=2$ and a double point if $d>2$), and each $\partial U_j$ has exactly three singular points.\end{thm}

\begin{proof} Let $d_\infty$ and $d_j$ denote the number of double points on the boundary of the corresponding $U$'s. Similarly, let $c_\infty$ and $c_j$ be the number of cusps. By the definition of Suffridge polynomials we have
$$d_\infty+\sum d_j=2d-4,\qquad c_\infty+\sum c_j=d-1.$$
Adding up, we get
$$d_\infty+c_\infty +\sum_{j=1}^{d-2} (d_j+c_j) = 3(d-2) + 1,\quad \text{i.e.} \quad (d_\infty+c_\infty-1) +\sum_{j=1}^{d-2} (d_j+c_j-3)=0.$$
For $d>2$, we have $d_\infty\geq 1$ (because otherwise there cannot be any bounded $U_j$).  By Lemma \ref{lem:3} all the terms  in the right equation are non-negative and hence must be zero.  For $d=2$, we have $d_\infty=0$ and therefore $c_\infty=1$.
\end{proof}

\section{Existence of Suffridge polynomials} \label{sec:extreme}

In this section we will prove Theorems \ref{thm:Sd} and \ref{thm:Sigma} (and therefore Theorem \ref{thm:D} in Introduction). At the end of the section we will comment on Suffridge's paper \cite{Suffridge72}.

\subsection{Construction}

First we note that Theorem \ref{thm:Sd} is equivalent to the following statement.\bigskip

\begin{thm}\label{thm:f1f2} Suppose $f\in (S^*_d)$ and $N$ be the number of double points in $f(\TT)$.  If $N<d-2$ then there exist $f_1\ne  f_2$ in $(S^*_d)$
such that
$$f=\frac12(f_1+f_2).$$\end{thm}
To construct $f_1$ and $f_2$, let us denote by
$$\{\zeta^\pm_j=\ee^{it_j^\pm}~|~j=1,\cdots,N\}\subset \TT$$
the set of preimages of the double points of $f(\TT)$ (with any particular choice of $\pm$).   We will need to find a  non-trivial polynomial
 $$r(z)=\sum_{j=2}^{d-1} a_jz^j$$ 
with the following properties:
\begin{itemize}
\item[({\bf R1})]\label{item:R1} $r'(z)$ is self-dual in $\Pi_{d-1}$; 
\item[({\bf R2})]\label{item:R2} $\displaystyle{\rm Re}\bigg(\frac{r(\zeta^+_j)-r(\zeta^-_j)}{(\zeta_j^+)^{(d+1)/2}}\bigg)=0$ for $j=1,\cdots,N$.
\end{itemize}
By Lemma \ref{lem:kappa} the geometric meaning of the second condition states that the line through the points $r(\zeta^+_j)$ and $r(\zeta^-_j)$ is parallel to the tangent line of $f(\TT)$ at the double point $f(\zeta^\pm_j)$. 

There are infinitely many polynomials $r$ satisfying the two conditions because ({\bf R2}) gives us $N<d-2$ homogeneous linear equations in $\RR^{d-2}$ while the (linear) space of polynomials satisfying ({\bf R1}) has $(d-2)$ real dimension. To write the equations explicitly, we can represent, for real numbers $c_\ell$'s,
$$r=\sum_{\ell=1}^{d-2} c_\ell\, r_\ell$$
where $\{r_\ell\}_{\ell=1,\cdots,d-2}$ is some basis in the real linear space,
 $$\{r(z)=\sum_{j=2}^{d-1}a_j z^j|\text{$r'$ is self-dual in $\Pi_{d-1}$}\}.$$ 
 For instance we can consider the basis
$$ \bigg\{\frac{z^2}{2}+\frac{z^{d-1}}{d-1},\quad i\left(\frac{z^2}{2}-\frac{z^{d-1}}{d-1}\right), \quad
\frac{z^3}{3}+\frac{z^{d-2}}{d-2},\quad i\left(\frac{z^3}{3}-\frac{z^{d-2}}{d-2}\right),\quad \dots\bigg\}.$$

For example, for $d=3$, the basis is simply $\{z^2\}$, and for $d=4$, the basis is
$$\bigg\{\frac{z^2}{2}+\frac{z^{3}}{3},i\left(\frac{z^2}{2}-\frac{z^{3}}{3}\right)\bigg\}.$$
For $d=4$, we consider $r=c_1 r_1+c_2 r_2$ with $c_{1,2}\in\RR$.  Assuming there is at most one ($<d-2$) double point, ({\bf R2}) gives at most one homogeneous linear equation for $c_1$ and $c_2$.  Obviously, there exists a nontrivial solution.

\begin{figure}\begin{center}
\includegraphics[width=0.7\textwidth]{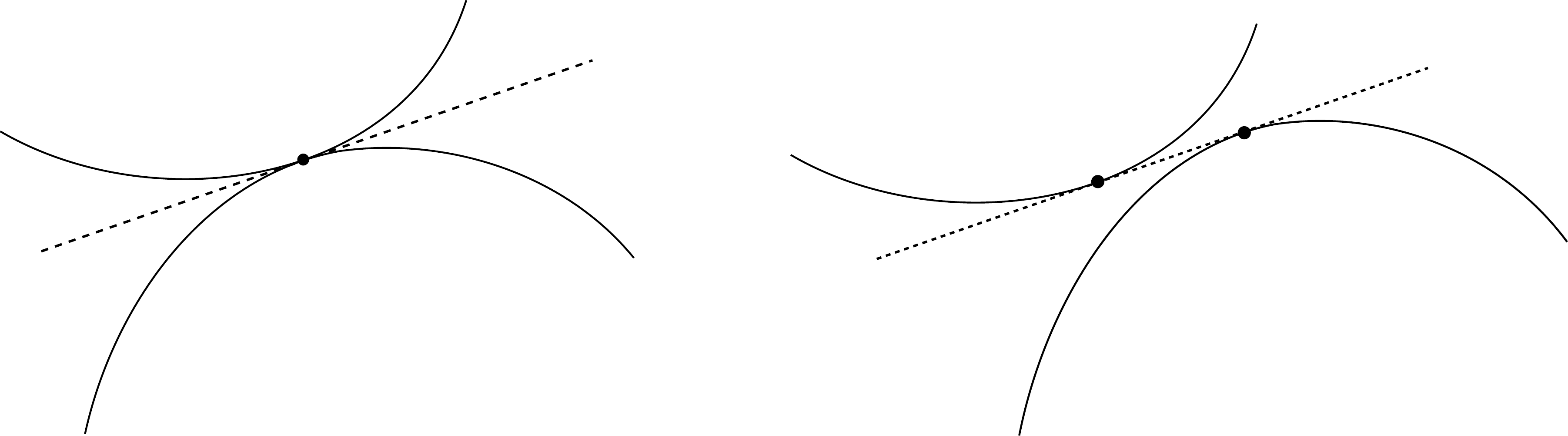}
\caption{\label{fig:split}$g_\delta(\TT)$ near a double point when $\delta=0$ (left) and $\delta\neq 0$ (right). The dots represent $g_\delta(\zeta_j^\pm)$.}
\end{center}\end{figure}
Let us define
$$g_\delta(z)=f(z)+\delta \, r(z)\quad (\delta\in \RR).$$
To prove Theorem \ref{thm:f1f2} it is enough to prove that there exists $\delta_0>0$ such that $g_\delta$ is {\em univalent} for all $\delta\in[-\delta_0,\delta_0]$. 
This will be done in the next subsection.

\subsection{Univalency}

First we describe how cusps of $g_\delta(\TT)$ move under the variation of $\delta$.
\bigskip
\begin{lemma}[Preservation of cusps]\label{lem:ZF} 
There exists $\delta_1>0$ such that, for all $\delta\in[-\delta_1,\delta_1]$, all the $d-1$ zeros of $g'_\delta$ are simple, on $\TT$, and continuously moving under the variation of $\delta$.
\end{lemma}
\begin{proof}
By ({\bf R1}), $g'_\delta(z)$ is self-dual in $\Pi_{d-1}$ for all $\delta$ and, therefore, the zeros of $g'_\delta$ are symmetric under the circular inversion with respect to $\TT$ (i.e. under $z\mapsto 1/z$).   Since the zeros move continuously over $\delta$ by Hurwitz's theorem, and since all the zeros of $g'_{\delta=0}=f'$ are simple and on $\TT$, we obtain the lemma.
\end{proof}
The next lemma describes what happens to the double points of $f(\TT)$ under the variation of $\delta$. 
\bigskip

\begin{lemma}[Splitting of double points]\label{lem:split-double} There exists $\delta_2>0$ such that, for all $\delta\in[-\delta_2,\delta_2]$ and for all $j=1,\cdots,N$, the tangent line at $g_\delta(\zeta^+_j)$ coincides with the tangent line at $g_\delta(\zeta^-_j)$.
\end{lemma}
\begin{proof} 
By ({\bf R1}), $g'_\delta$ is self-dual in $\Pi_{d-1}$, and Lemma \ref{lem:kappa} implies that the unit tangent vector of $g_\delta(\TT)$ at $g_\delta(\zeta)$ is given by $\pm i \zeta^{(d+1)/2}$ when $\zeta$ is {\em not} a critical point of $g_\delta$.  We note that, by \eqref{eq:one}, $(\zeta_j^+)^{(d+1)/2}=\pm(\zeta_j^-)^{(d+1)/2}$ for some sign.
Then the condition ({\bf R2}) means that the tangent lines of the curve $g_\delta(\TT)$ at two points, $g_\delta(\zeta_j^+)$ and $g_\delta(\zeta_j^-)$, coincide because the difference, $g_\delta(\zeta_j^+)-g_\delta(\zeta_j^-)=\delta (r(\zeta_j^+)-r(\zeta_j^+))$, is parallel to the tangent line at $g_\delta(\zeta^\pm_j)$ (see Figure \ref{fig:split} for an illustration). 

Using Lemma \ref{lem:ZF}, $\zeta^\pm_j$'s are distinct from all the critical points of $g_\delta$ for $\delta$ small enough, and the lemma follows.
\end{proof}

Next we show that the curve $g_\delta(\TT)$ is ``locally simple''.  The main idea is that $g_\delta(\TT)$ has a constant conformal curvature (due to Lemma \ref{lem:kappa} and the self-dual property of $g'_\delta$).  Such rigidity makes it impossible to create a ``small'' self-intersecting arc.

\bigskip
\begin{lemma}[Local injectivity]\label{lem:loc-simple} Let $g$ be a polynomial  such that $g(\TT)$ has a constant conformal curvature $\kappa>0$ away from the cusps. There exists $\epsilon_1>0$ such that, if $I\subset\TT$ is an arc with $|I|<\epsilon_1$ then $g(I)$ is a simple arc.
\end{lemma}
\begin{proof}
Since there are a finite number of cusps, one can choose $\epsilon_1\in(0,\pi/\kappa)$ so small that $g(I)$ with $|I|<\epsilon_1$ can have at most single cusp.

If $g(I)$ is not a simple curve, there exists $\zeta_1\neq \zeta_2$ in $I$ such that $g(\zeta_1)=g(\zeta_2)$. Let $I_0\subset I$ be the arc with the endpoints at $\zeta_1$ and $\zeta_2$.  We may assume that $g(I_0)$ is a simple closed curve; if not one can choose different $\zeta_1$ and $\zeta_2$.   Let us denote $P=g(\zeta_1)=g(\zeta_2)$.

There are two possibilities: either $g(I_0)$ has i) a single cusp singularity or ii) none.
For the case i), denoting the cusp point by $Q$, $g(I_0)$ consists of {\em two smooth arcs} that share the two endpoints $P$ and $Q$ (leftside in Figure \ref{fig:proof}).   For the case ii) there is only one smooth arc that starts and ends at a single point, $P$ (rightside in Figure \ref{fig:proof}).
Below, we show that neither case is possible.

Case i): Let us denote the two smooth arcs of $g(I_0)$ by $A_1$ and $A_2$ (see Figure \ref{fig:proof}):
$$g(I_0)=A_1\cup A_2$$ such that $A_1\cap A_2=\{Q,P\}$.  
As one moves along $A_1$ or $A_2$ the angle of the tangent changes monotonically.
Since $|I_0|\leq \pi/\kappa$ the angle of the tangent cannot vary more than $\pi$.
Therefore the tangent direction is always directed {\em away} from the straight line (e.g. the dashed line in Figure \ref{fig:proof}) that passes through the cusp point towards the direction of the cusp.  This means that the whole curve $A_1$ cannot intersect the straight line except at the cusp point.
The same holds for $A_2$ and this is a contradiction (i.e. the illustrated situation cannot happen).

Case ii): As one moves along $g(I_0)$ to either direction from $P$ the angle of the tangent direction changes linearly.
Since $|I_0|\leq \pi/\kappa$ the angle of the tangent cannot vary more than $\pi$.
Therefore the tangent direction is always directed {\em away} from the straight line (e.g. the dashed line in Figure \ref{fig:proof}) that is tangent to the initial direction at $P$.  Therefore the distance from the straight line increases monotonically along the curve  $g(I_0)$ and this is a contradiction.
\begin{figure}
\begin{center}
\includegraphics[width=0.5\textwidth]{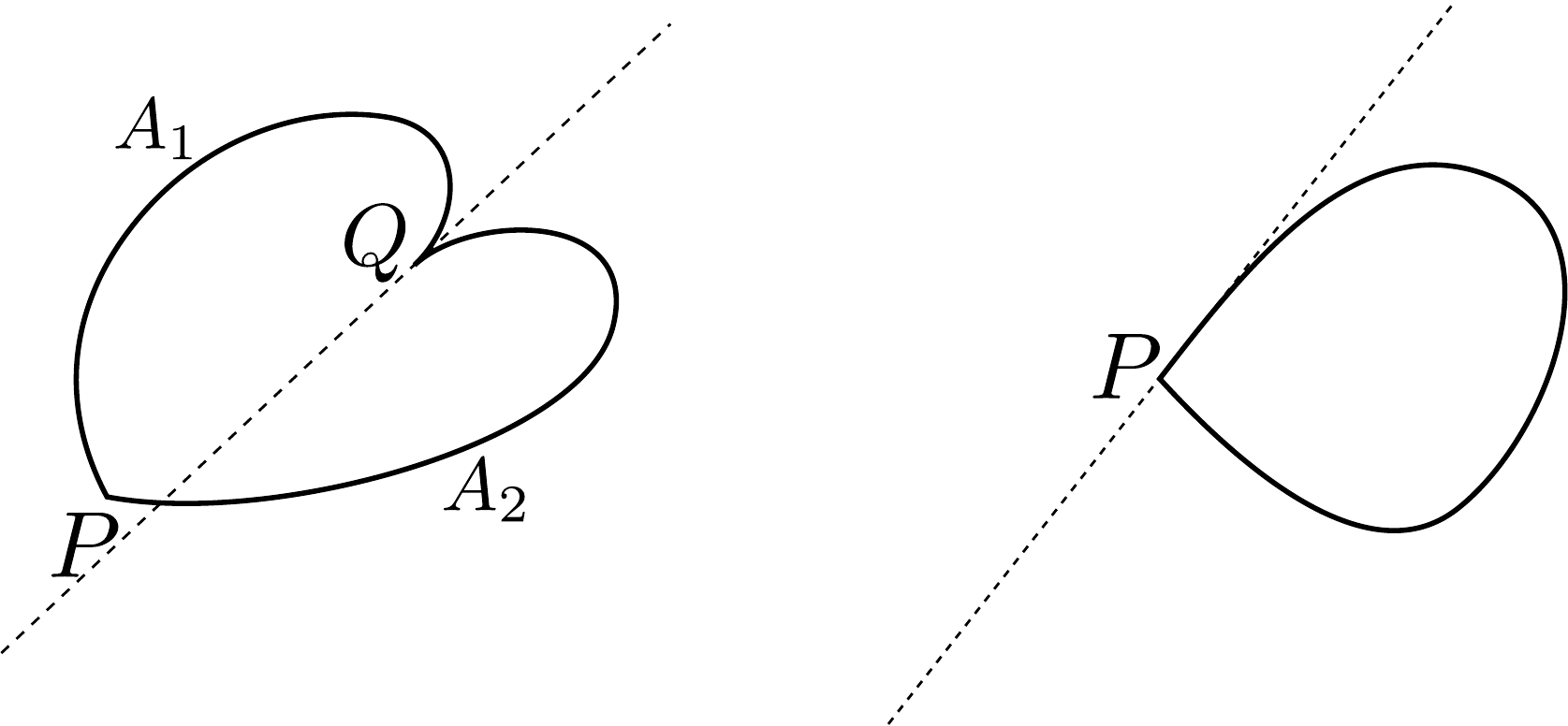}
\caption{\label{fig:proof} Both cases are impossible as they require a large curvature}
\end{center}
\end{figure}
\end{proof}
Using the similar argument and Lemma \ref{lem:split-double}, we show that, near $g_\delta(\zeta_j^\pm)$, $\zeta^\pm_j=\exp({i t^\pm_j})$, the only possible self-intersection is a single double point.
\bigskip
\begin{lemma}[Self-intersection near double points]\label{lem:global-simple} There are sufficiently small $\epsilon_2>0$ and $\delta_2>0$, such that, for all $\delta\in[-\delta_2,\delta_2]$, the possible intersection between $g_\delta(I^+_j)$ and $g_\delta(I^-_j)$ is at most a single double point (for each $j$) where  
$$I^\pm_j=\{\ee^{i t}~ |~ t\in[t^\pm_j-\frac{\epsilon_0}{2},t^\pm_j+\frac{\epsilon_0}{2}]\}\subset\TT.$$
\end{lemma}
\begin{proof} 
We choose $\delta_2$ by the one in Lemma \ref{lem:split-double}.  By choosing $\epsilon_2<\pi/(2\kappa)$ small enough we can make $I^+_j\setminus\{g_\delta(\zeta^+_j)\}$ and $I^-_j\setminus\{g_\delta(\zeta^-_j)\}$ to be cusp-free and to be confined in the separate half-planes divided by the straight line that is simultaneously tangent at $g_\delta(\zeta^+_j)$ and $g_\delta(\zeta^-_j)$.     The confinement in a half-plane follows from the same argument used in Lemma \ref{lem:loc-simple} (that is, by the constraint on the curvature, the line segment $I^\pm_j$ cannot turn itself to intersect the straight line except at $g_\delta(\zeta_j^\pm)$).
\end{proof}

\begin{proof}[Proof of Theorem \ref{thm:f1f2}]
We choose $N_0 \geq 4\pi/\min\{\epsilon_1,\epsilon_2\}$ where $\epsilon_1$ and $\epsilon_2$ are from Lemma \ref{lem:loc-simple} and Lemma \ref{lem:global-simple}.  We also choose $\delta_0\leq\min\{\delta_1,\delta_2\}$ where $\delta_1$ and $\delta_2$ are from Lemma \ref{lem:ZF} and Lemma \ref{lem:split-double}.

From Lemma \ref{lem:loc-simple}, $g_\delta(J_\ell)$ is a simple arc for
$$J_\ell:=\{ \ee^{i t} ~|~ t\in \Big[2\pi\frac{ \ell}{N_0},2\pi\frac{\ell+1}{N_0}\Big]\}\subset\TT,\quad \ell=0,1,\cdots,N_0-1.$$
For $\delta=0$, there are three possibilities for a given pair $(J_\ell,J_k)$ of arcs: $g_0(J_\ell)\cap g_0(J_k)=\emptyset$; $\ell$ and $k$ are {consecutive} (in such case, we get $g_\delta(J_{k-1})\cap g_\delta(J_k)=\{g_\delta(\ee^{2\pi i k/N_0})\}$); or $g_\delta(J_\ell)$ and $g_\delta(J_k)$ share a single double point.   For the pair $(J_\ell,J_k)$ in the first case, by the continuity of $g_\delta$ over $\delta$, we have
$g_\delta(J_\ell)\cap g_\delta(J_k)=\emptyset$ for $\delta<\delta_0$ for a sufficiently small $\delta_0$.

So we proved that, there exists $\delta_0>0$ such that, for all $\delta\in[-\delta_0,\delta_0]$, the only possible self-intersections of $g_{\delta}(\TT)$ are a finite number of double points.

Let us denote by $\Omega_\delta$ the simply-connected bounded domain enclosed by $g_\delta(\TT)$.  At $\delta=0$, from the univalency of the map $f=g_{0}$, for any point $p\in\Omega_{0}$ the loop $g_{0}(\TT)$ has the winding number one, and winding number zero for any point $p\notin\clos\Omega_{0}$.

Since the only self-intersections of $g_\delta(\TT)$ are double points, there is no ``crossing'' of the boundary $g_\delta(\TT)$ as one varies $\delta$ in $[-\delta_0,\delta_0]$.  
It implies that the winding numbers of $g_\delta(\TT)$ at the regions $\Omega_\delta$ and $\CC\setminus\clos\Omega_\delta$ are preserved under the variation of $\delta$; they are fixed to one and zero respectively.
By the argument principle, it means that $g_\delta:\DD\to \Omega_\delta$ is univalent, and this proves Theorem \ref{thm:f1f2}.
\end{proof}

\subsection{Proof in the case of \texorpdfstring{$(\Sigma_d)$}{Sigmad}}

Since the proof of Theorem \ref{thm:Sigma} is similar to the proof of Theorem \ref{thm:Sd}, we will briefly outline the argument.
We need to verify the following statement:
\vspace{0.2cm}
\begin{thm}
Suppose $f\in (\Sigma_d^*)$ has exactly $N<d-2$ double points.  Then there exist $f_1\neq f_2$ in $(\Sigma_d^*)$ such that  $$f=\frac{f_1+f_2}2.$$
\end{thm}

Let $f(\zeta_j^\pm)$'s ($j=1,\cdots,N$) be the double points of the curve $f(\TT)$.
Since $f'(1/z)$ is self-dual in $\Pi_{d+1}$, we consider the perturbation $g_\delta=f+\delta r$ ($\delta\in\RR$) of $f$ such that
 $$r(z)=\sum_{j=1}^{d-1} a_j/z^j$$ 
satisfies:

({\bf R1}) $r'(1/z)$ is self-dual in $\Pi_{d+1}$ (hence $a_{d-1}=0$);

({\bf R2}) $\displaystyle{\rm Re}\frac{r'(\zeta_j^+)-r'(\zeta_j^-)}{(\zeta_j^+)^{(1-d)/2}}=0$, for $j=1,\cdots,N$.

The latter condition says, by Lemma \ref{lem:kappa}, that the image of $\zeta^\pm_j$ under $g_\delta$ is such that their tangent lines coincide (see Figure \ref{fig:split}).
Since the (real-)dimension of the space of $r$ that satisfies ({\bf R1}) is $d-2$, there exists (infinitly many) $r$ that satisfies ({\bf R2}).

Then we show that $g_\delta$ is univalent on $(\clos\DD)^c$ for sufficiently small $\delta$.  All the statements (Lemma \ref{lem:ZF}, Lemma \ref{lem:split-double}, Lemma \ref{lem:loc-simple}, Lemma \ref{lem:global-simple}) are true and the proofs are same.

\subsection{Remark}\label{app:Suffridge}

If $q=q(z)$ is analytic in $\DD$ and $\zeta\in\TT$ then we define the function
\begin{equation*}
	(D_\zeta q)(z)=\frac{q(\zeta z)-q(\overline \zeta z)}{z(\zeta-\overline \zeta)},\qquad z\in\DD.
\end{equation*}
Given $n>2$, the Suffridge class $(P_n)$ consists of polynomials 
\begin{equation*}
	q(z)=z+b_2z^2+\cdots+b_nz^n
\end{equation*}
such that 
\begin{equation*}
 \zeta^{2n+2}=1,\quad \zeta\neq \pm1\quad\Rightarrow\quad D_\zeta q \text{ does not vanish in $\DD$.}
\end{equation*}

{\bf Claim} { (Theorem 1 + Theorem 3 + Theorem 6 in \cite{Suffridge72}) If $p$ is an extreme point of $(P_n)\subset\Pi_n$, then the polynomial
\begin{equation*}
  f(z)=\frac{n+1}{n} p(z)-\frac{1}{n} z\,p'(z)
\end{equation*}
is univalent in $\DD$ and $f(\DD)$ has $n-2$ double points.}

The most difficult part of this combination of three theorems is Theorem 3 in \cite{Suffridge72}, which can be regarded as a univalency criterion for $f$.
In terms of $f$, Theorem 3 can be restated as follows.

{\em Let a polynomial
$$f(z) = 1+ a_2 z^2+\cdots + \frac{1}{n}z^n $$
be such that $f'$ is self-dual in $\Pi_{n-1}$.  Then $f$ is univalent in $\DD$ if and only if 
\begin{equation}\label{eq:univ-crit}
	\zeta^{2n+2}=1,\quad \zeta\neq \pm1\quad\Rightarrow\quad D_\zeta f \text{ does not vanish in $\DD$.}
\end{equation}
}
For polynomials, this result is an ambitious improvement over Dieudonn\'e's univalence criterion \cite{Dieudonne} which states that {\it an analytic function $f$ in $\DD$ is univalent if and only if $D_\zeta f$ does not vanish in $\DD$ for all $\zeta\in\TT$}.

Unfortunately, Suffridge's proof in \cite{Suffridge72} is somewhat sketchy and we have been unable to follow and verify some of the arguments.  Perhaps some similar arguments are better explained in his previous papers.  Anyway, we presented in this Section a more direct self-contained proof of the existence of Suffridge polynomials which uses neither the univalence criterion \eqref{eq:univ-crit} nor the Suffridge classes $(P_n)$. At the same time, we want to emphasize that the main ideas (the use of constant curvature and Krein-Milman theorem) are from Suffridge's paper.

\section{Inscribed cardioid theorem}\label{sec:packing}

\subsection{Deltoid-like and cardioid-like curves}\label{sec:inscribed-card}

\begin{defn}
A  Jordan curve in the plane is {\it cardioid-like} if it is smooth and has positive curvature with respect to counterclockwise orientation except for a single (inward) cusp singularity.
\end{defn}

\bigskip\begin{defn}
A Jordan curve is {\it deltoid-like} if it is smooth except for three outward cusps and if (at least) one of the arcs between the cusp points has negative curvature and the total variation of the tangent direction along this arc is less than $\pi$.  It $T$ is a deltoid-like curve, then we will denote such an arc by $ T_\text{conc}$. 
\end{defn}

Sometimes, we will call {\it cardioid-like} or {\it deltoid-like} for the corresponding (bounded) Jordan domains.

\bigskip
\begin{prop}\label{prop} (i) Let $\Omega$ be an unbounded extreme quadrature domain of order $d\geq 2$.  The interior of $\Omega^c$ has $d-1$ components bounded by deltoid-like curves.

(ii) Let $\Omega$ be a bounded extreme quadrature domain of order $d\geq 2$. The interior of $\Omega^c$ has $d-1$ components. The boundary of the unbounded component is cardioid-like, and  the boundaries of the bounded components are deltoid-like. 
\end{prop}

\begin{proof} 
All facts follow from Corollary \ref{cor:kappa} and Theorems \ref{thm:geom2} and \ref{thm:geom1}.

To show that, for a bounded component of $(\clos\Omega)^c$, at least one of the arc has variation of the tangent direction less than $\pi$, we use the equation \eqref{eq:fullrotate}: the total variation of the unit tangent vector over all smooth arcs of a deltoid-like curve is 
$$\int| k|ds=\left|\int kds\right|=\pi,$$
so not only one but all three arcs of a deltoid-like curve have angular variation less than $\pi$.
\end{proof}

The following is the main result of this section.

\bigskip
\begin{thm}\label{thm:card-in-delt} Let $C$ be  a cardioid-like curve and $T$ a deltoid-like curve.  
Then we can inscribe $C$ in $T$ so that $C$ intersects all three sides of $T$ and intersects $T_\text{conc}$ twice. \end{thm}

The proof of the theorem is in the next two subsections.

\subsection{Three lemmas}

\begin{lemma}[Separation Lemma]\label{lem:sep}
Let $\Omega$ be a Jordan domain and $A\subset\clos\Omega$ be a closed connected subset such that $\Gamma=A\cap\partial \Omega$ is a simple arc with the endpoints at $P$ and $Q$.  We further assume that any point in $\Gamma_o=\Gamma\setminus\{P,Q\}$ has a neighborhood that does not intersect $\clos\Omega\setminus A$. 
Then $\partial A\cap\Omega$ has a connected component whose closure contains $P$ and $Q$.
\end{lemma}

\begin{proof} 

Since $A$ is connected, $(\clos\Omega)\setminus A$ has only simply-connected components.
Let $B$ be the simply connected component of $(\clos\Omega)\setminus A$ that contains $\partial\Omega\setminus\Gamma$.
By this definition $\partial B$ contains $P$ and $Q$.  
Observe that $\partial B$ does not intersect $\Gamma_o$ because of the assumption that any point in $\Gamma_o$ has a neighborhood in $$(\clos\Omega\setminus A)^c=(\clos\Omega)^c \sqcup A,$$
which lies in the complement of $B$.  

Now consider the following decomposition:
$$ \partial B = (\partial B\cap\Omega)~\sqcup~(\partial\Omega\setminus\Gamma)~ \sqcup~\{P,Q\}.$$
Since $B$ is simply-connected, $\partial B$ is doubly-connected.  That is, if $\{P,Q\}\subset \partial B$ and $P\neq Q$ then there are two separate paths that connect $P$ and $Q$.  We know that $\partial\Omega\setminus\Gamma$ gives one of them.  Therefore, the other path must exist in $\partial B\cap\Omega$.  It will be enough to show that $\partial B\cap\Omega\subset\partial A\cap\Omega$.

From the definition of $B$, we have
$$\partial B ~\subset~ \partial (\clos\Omega\setminus A)~\subset~ \partial\Omega \cup \partial A,$$
and we get
$  \partial B\cap\Omega ~\subset~ \partial A\cap\Omega$. 
\end{proof} 

\begin{remark}
Below we will consider the relative position of a geometric object (e.g. cardioid) with respect to the other geometric object (e.g. $\Gamma$).   The expression ``unique cardioid'' in the lemma implies that, fixing the latter geometric object in space, there exists a {\em unique transformation} of a cardioid where the transformation consists of {\em rotation, translation and uniform scaling}.
\end{remark}
\bigskip

\begin{lemma}\label{lem:2tangents}  Let $\Gamma$ be a simple smooth arc with non-vanishing curvature and with the total variation of the angle of tangents being less than $\pi$.     Given two points $p$ and $q$ on $\Gamma$ there exists a unique cardioid that tangentially intersects $\Gamma$ at $p$ and $q$, such that the two arcs that connects $p$ and $q$ respectively to the cusp of the cardioid and the subarc of $\Gamma$ with its two end points at $p$ and $q$ form a deltoid-like curve with only concave arcs.
\end{lemma}

\begin{figure}[ht]
\begin{center}
\includegraphics[width=0.9\textwidth]{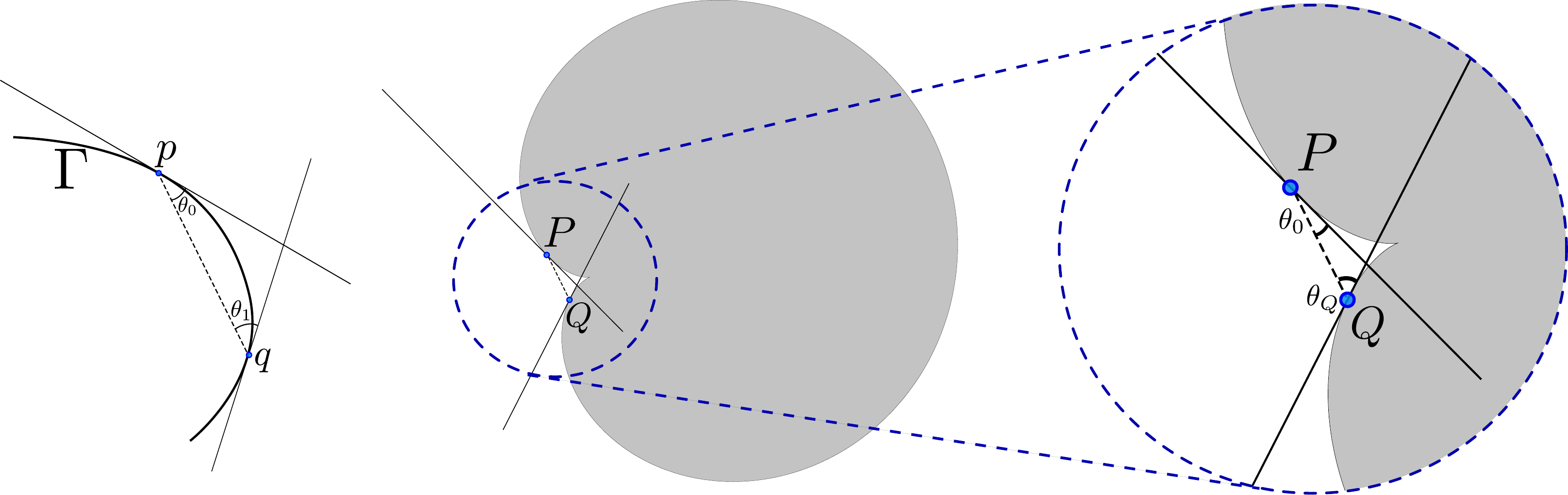}
\caption{\label{fig:2tangents} Proof of Lemma \ref{lem:2tangents}.  The rightmost picture shows the zoom-in.  }
\end{center}
\end{figure}

\begin{proof}
Let $p$ and $q$ are two points on $\Gamma$.  Let $\overline{pq}$ be the straight line segment between $p$ and $q$. We denote the angle between the tangent line at $p$ and $\overline{pq}$ by $\theta_0$, and the angle between the tangent line at $q$ and $\overline{pq}$ by $\theta_1$, see Figure \ref{fig:2tangents}.  By the definition of $\Gamma$, we have $\theta_0+\theta_1<\pi$ and, therefore, the two tangent lines at $p$ and $q$ meet on the convex side of $\Gamma$.  Note that the tangent lines at $p$ and $q$ divides the plane into four regions and $\Gamma$ sits in (the closure of) one of them, that we denote by $\Sigma_{pq}$.

We claim that, for any positive angles $\theta_0$ and $\theta_1$ such that $\theta_0+\theta_1<\pi$, there exists two points, say $P$ and $Q$, on the cardioid such that the triangle formed by the two tangent lines at $P$ and $Q$ and $\overline{PQ}$ has the following two properties: i) the triangle sits {\em outside} the cardioid, ii) the triangle has the angles $\theta_0$ and $\theta_1$ respectively at its vertices $P$ and $Q$.  Then, the existence part of the lemma follows because one can simply match $P$ and $Q$ to the point $p$ and $q$ by some ``rigid transformation + uniform scaling'' of the cardioid.

We show the above claim geometrically.  Pick a point $P$ on the cardioid.  Consider two lines that pass through $P$: one is the tangent line at $P$, and the other line is rotated by the angle $\theta_0$ (the dashed line in the rightmost picture in Figure \ref{fig:2tangents}).  Let $Q$ be the {\em first} intersection of the latter line with the cardioid such that $\overline{PQ}$ sits outside the cardioid.  Such $Q$ may not exist depending on the position of $P$.   When $Q$ exists, we denote the angle between the tangent at $Q$ and $\overline{PQ}$ by $\theta_Q$, see Figure \ref{fig:2tangents}.

As one moves $P$ along the cardioid towards the cusp point, one can obverse that the angle $\theta_Q$ varies over $[0,\pi-\theta_0)$; the angle is $0$ when $\overline{PQ}$ is tangential at $Q$, and approaching $\pi-\theta_0$ when $P$ approaches the cusp point.  Therefore, there exists $P$ such that $\theta_Q=\theta_1$, and the claim is proven.

The uniqueness comes from the observation that $\theta_Q$ is monotonic from $0$ to $\pi-\theta_0$ as $P$ moves towards the cusp point.  The monotonicity can be obtained from the following argument.  As $P$ moves towards the cusp, $Q$ also moves towards the cusp, which means that the tangent line at $Q$ has the slope that changes monotonically.  By the positivity of the curvature, the slope of the ``dashed line'' (the tangent line at $P$ rotated by $\theta_0$) also changes monotonically.  Above two ``monotonicity'' gives the monotonicity of the angle $\theta_Q$.      
\end{proof}

\begin{lemma}\label{lem:touch} Let $T$ be a deltoid-like curve.  Given $p,q\in T_\text{conc}$, Lemma \ref{lem:2tangents} says that there is a unique cardioid that tangentially intersect $T_\text{conc}$ at $p$ and $q$. Denote that cardioid by $C_{pq}$.  Define 
$$Z_\text{out}=\{(p,q): p,q\in[\partial T]_\text{conc}; {\rm Int}\,C_{pq}\cap {\rm Ext}\,T\neq \emptyset\}.$$  
(Geometrically, this set corresponds to the case where the cardioid does {\em not} sit inside the deltoid-like curve.)
If $(p,q)\in\partial Z_\text{out}$ then $C_{pq}$ tangentially intersects at least one of the sides other than $ T_\text{conc}$.
\end{lemma}
\begin{proof}  
$Z_\text{out}$ is an open set, because $C_{pq}$ is continuous in $(p,q)\in\partial T_\text{conc}\times\partial T_\text{conc}$ and, therefore, if ${\rm Int}\,C_{pq}\cap{\rm Ext}\,T\neq\emptyset$ then there exists a neighborhood of $(p,q)$ such that the intersection remains non-empty.
Therefore, if $(p,q)\in\partial Z_\text{out}$, then we have
$${\rm Int}\,C_{pq}\cap {\rm Ext}\,T = \emptyset.$$
Also, if $(p,q)\in\partial Z_\text{out}$, we also claim that
$$ C_{pq}~\cap~(T\setminus  T_\text{conc}) ~\neq~ \emptyset.$$  
If this is not true, then there exists an open neighborhood of $(p,q)\in T_\text{conc}\times T_\text{conc}$ such that the equation continues to not hold (again by the continuity of the maping $(p,q)\mapsto C_{pq}$). In such neighborhood ${\rm Int}\,C_{pq}\cap{\rm Ext}\,T$ stays empty and this contradicts the fact that $(p,q)\in\partial Z_\text{out}$. 
\end{proof}

\subsection{Proof of Theorem }

Let the two end points of $T_\text{conc}$ be $p_L$ and $p_R$.  Let us consider a parametrization of $T_\text{conc}$ by $\gamma:[0,1]\to T_\text{conc}$ such that $p_L=\gamma(0)$ and $p_R=\gamma(1)$ so that one can say $p<q$ or $[p,q]$ in terms of the corresponding preimages on $[0,1]$ through the bijection $\gamma$. 
Then one can observe that the set $Z_\text{out}$ contains $\{p_L\}\times (p_L,p_R]$  and $[p_L,p_R)\times\{p_R\}$.    And $Z_\text{out}$ does not include the diagonal point $(p,p)$ for all $p\in T_\text{conc}$.  Applying Lemma \ref{lem:sep} with 
$$ \clos\Omega= \{(p,q)\in T_\text{conc}\times T_\text{conc}~|~p<q\} \quad\text{and}\quad A=\clos Z_\text{out},
$$ 
we obtain that $\partial Z_\text{out}$ connects $(p_L,p_L)$ and $(p_R,p_R)$. 

Let us denote the other two sides of $T$ other than $T_\text{conc}$ by $T_1$ and $T_2$ such that $p_L\in T_1$ and $p_R\in T_2$. We define
$$
A_{j}=\left\{(p,q)\in \partial Z_\text{out} \,\big|\, C_{pq}\cap T_j\neq\emptyset\right\} \text{ ~for~ j=1,2}.
$$
Both $A_1$ and $A_2$ are closed subsets of $\partial Z_\text{out}$ and $A_1\cup A_2=\partial Z_\text{out}$ by Lemma \ref{lem:touch}.  On the other hand, $A_1$ contains $(p_L, p_L)$ while $A_2$ contains $(p_R,p_R)$.  Since $\partial Z_\text{out}$ connects $(p_L,p_L)$ and $(p_R,p_R)$, there must exist a point in $\partial Z_\text{out}$ that belongs to $A_1\cap A_2$.  This proves Theorem \ref{thm:card-in-delt}.

\subsection{Inscribed circles}

We will also inscribe circles in deltoid-like curves.
\bigskip

\begin{thm}\label{thm:circ-in-delt} Let  $T$ be a  deltoid-like curve. Then we can inscribe a circle in $T$ so that the circle intersects all three sides of $T$. \end{thm}

\begin{proof}
The details of the proof are parallel to (and simpler than) the proof of Theorem \ref{thm:card-in-delt}, which will follow.  First, given a deltoid-like curve $T$, there exists $T_\text{conc}$ by the definition.  Given a point $p$ on $T_\text{conc}$ and a positive number $r$, there exists a unique circle of radius $r$ that intersects at $p$ such that the circle is on the {\em convex} side of $T_\text{conc}$ (this corresponds to Lemma \ref{lem:2tangents} for the case of Theorem \ref{thm:card-in-delt}).  For each $p$ there exists the minimal radius, say $r_0(p)$, such that the above circle of radius $r_0(p)$ intersecting at $p$ intersects at least one of sides of $T$ other than $T_\text{conc}$ -- the set $\{(p,r_0(p))\}$ corresponds to $\partial Z_\text{out}$ where $Z_\text{out}$ is defined in Lemma \ref{lem:touch}.  Then each point $p\in T_\text{conc}$ can be identified with the either of the two sides (other than $T_\text{conc}$) that the corresponding circle intersects; let us define $A_1$ and $A_2$ to be the set of points $p$ such that the corresponding circle intersects with the side one and the other, respectively.  Since $A_1$ and $A_2$ are both closed sets in $T_\text{conc}$, and $T_\text{conc}=A_1\cup A_2$, there exists an element in $A_1\cap A_2$.  This proves our theorem.   
\end{proof}

\bigskip

\section{Proof of the sharpness results}\label{sec:main-proof}

In this section we will prove Theorems \ref{thm:UQD} and \ref{thm:BQD}. In each case we  construct a union of disjoint quadrature domains such that the interior of their complement has the required number of connectivity components. This number will be equal to the connectivity of a single quadrature domain if we slightly perturb the picture by means of a Hele-Shaw flow. The perturbation procedure is explained in detail in Section 5 of \cite{1stpaper}. It is crucial that the location and multiplicities of the poles do not change as a result of the perturbation.

\subsection{Unbounded quadrature domains}

We first construct unbounded quadrature domains $\Omega$ such that all nodes are finite (the first part of Theorem \ref{thm:UQD}).

The case $n=d$ was covered in \cite{1stpaper}, Lemma 5.1.
We will consider the case $n<d$ and find $\Omega$ such that
$$\conn(\Omega) =d+n-1.$$ 
Given a partition 
$$d=\mu_1+ \cdots+\mu_n$$
we assume, without loss of generality, that $\mu_1\ge 2$. 
Then we choose an extreme BQD $\Omega_1$ of order $\mu_1$, and we inscribe $\Omega_1$ in an open round disc $U_1$ so that
$$\# (\partial U_1\cap \partial \Omega_1)\ge2,$$
see the first configuration in Figure \ref{fig:initialUQD} for the case $\mu_1=2$ where $\Omega_1$ is just a cardioid.  By Theorem \ref{thm:geom1} there are at least $\mu_1$ components in the interior of the complement of the (disjoint) union of the quadrature domains $\Omega_1$ and $\Omega_\infty:= \inte U^c_1$.   

We are done if $n=1$: a slight Hele-Shaw perturbation transforms $\Omega_1\sqcup\Omega_\infty$ into a UQD $\Omega$ with a single node  such that the order of $\Omega$ is given by $\mu_1=d$ and $\conn \Omega\ge \mu_1=d$.

\begin{figure}[ht]\begin{center}
\includegraphics[width=0.6\textwidth]{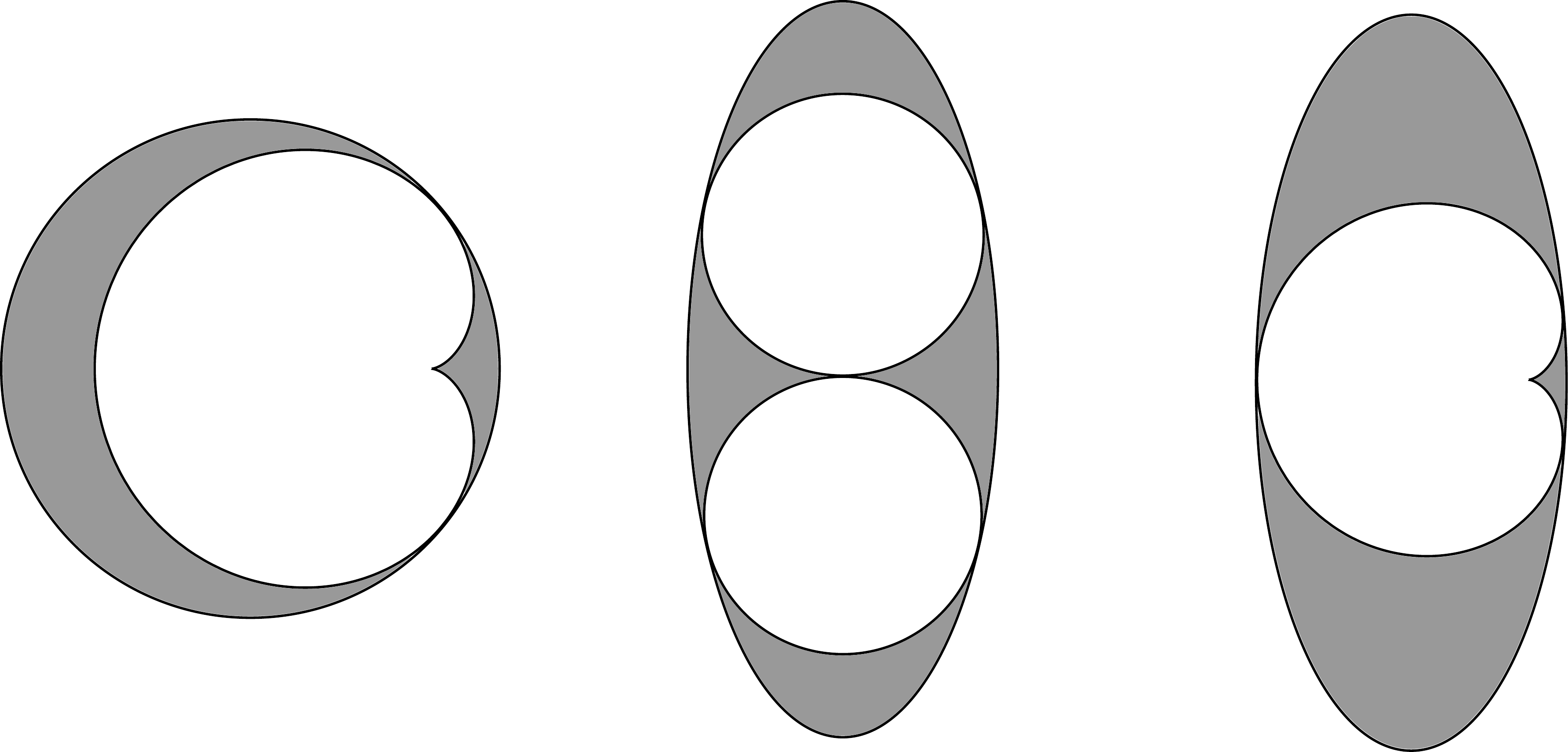}
\caption{Initial configurations for UQD\label{fig:initialUQD}}
\end{center}\end{figure}

If  $n>1$ then we proceed as follows.
By Proposition \ref{prop}, at least one of the components of the open set 
$U_1\setminus \clos \Omega_1$ is bounded by a deltoid-like curve.  (For example, the cusp of the cardioid-like boundary of $\Omega_1$ and the two intersections, $\partial\Omega_1\cap\partial U_1$, make the three cusps of the deltoid-like curve.)
Let us choose such a component and call it $U_2$. If $\mu_2=1$ then by Theorem \ref{thm:circ-in-delt} we can find an open disc $\Omega_2$ inside $U_2$ such that
$$\#(\partial U_2\cap \partial \Omega_2)\ge3.$$
If $\mu_2>1$, then  we use Theorem \ref{thm:card-in-delt} and find an extreme BQD $\Omega_2$ of order $\mu_2$
inside $U_2$ such that
$$\#(\partial U_2\cap \partial \Omega_2)\ge4.$$
In either cases we get disjoint quadrature domains $\Omega_1,\Omega_2$ and $\Omega_\infty$ such that the interior of $(\Omega_1\sqcup\Omega_2\sqcup\Omega_\infty)^c$ has at least $\mu_1+\mu_2+1$ components, and at least one of these components
has deltoid-like boundary.

After  we repeat the same process $n-2$ times more, we end up with $n$ disjoint BQDs $\Omega_j$, $j=1,\cdots,n$, such that each $\Omega_j$ has a single node of multiplicity $\mu_j$ and the number of components in the interior of
$$\bigg(\Omega_\infty\cup \bigsqcup_{j=1}^n \Omega_j\bigg)^c$$
is at least 
$$\mu_1+\sum_{j=2}^n (\mu_j+1)=d+n-1.$$
A slight Hele-Shaw perturbation will give us an unbounded quadrature domain $\Omega$ with all the required properties.

Let us now turn to the second part of Theorem \ref{thm:UQD}.  We denote the multiplicity of the node at $\infty$ by $\mu_n$.  The multiplcities of the other nodes are $\mu_1,\cdots,\mu_{n-1}$.

First we consider $\mu_n\ge 2$.  We choose an extreme UQD $\Omega_n$ of order $\mu_n$. By Proposition \ref{prop} there exist $\mu_n-1$ number of components in $(\clos\Omega_n)^c$ with deltoid-like boundary.   If $n=1$ then we are done.  A slight Hele-Shaw flow will give $\Omega$ with connectivity $\mu_n-1=d-1$.
If $n>1$ then we proceed as in the previous case with only finite nodes.  We end up with $n-1$ disjoint BQDs $\Omega_1,\cdots,\Omega_{n-1}$ with multiplcities $\mu_1,\cdots,\mu_{n-1}$ such that the number of components in the interior of 
$$\bigg( \bigsqcup_{j=1}^{n} \Omega_j\bigg)^c$$
is at least 
\begin{equation*}
  \mu_n-1 + \sum_{j=1}^{n-1} (\mu_1+1)=d+n-2.
\end{equation*}
A slight Hele-Shaw perturbation will give us an unbounded quadrature domain $\Omega$ with all the required properties.

Next we consider $\mu_n=1$.  We choose an elliptical UQD $\Omega_n$ with a large eccentricity.   The reason for the latter condition can be explained by Figure \ref{fig:initialUQD}.  In the second and the third pictures, we show the inscription of two disks and a cardioid in an ellipse.   Such configurations are possible if the ellipse is ``thin'' enough.   Moreover, the third configuration is possible for any cardioid-like curve instead of just the cardioid.

If all the finite nodes are simple ($n=d$), we use the second configuration of Figure \ref{fig:initialUQD}; we choose the two quadrature domains $\Omega_1$ and $\Omega_2$ by disks and inscribe them in $(\clos\Omega_n)^c$.  There are four components in
 $$U_1=\inte[(\Omega_n\sqcup\Omega_1\sqcup\Omega_2)^c].$$   If $d=3$ then a slight Hele-Shaw perturbation will give $\Omega$ with connectivity $4=d+1$, as desired. 
If $d>3$ then, by Theorem \ref{thm:circ-in-delt}, we choose a disk $\Omega_3$ to inscribe in one of the deltoid-like components of $U_1$ (there are two).    Iterating this process we get $\Omega_j$, $j=3,\cdots,n-1$ such that the number of components in the interior of $(\bigsqcup_{j=1}^{n}\Omega_j)^c$ is  
\begin{equation*}
  4+\sum_{j=3}^{n-1} 2 = 2d-2. 
\end{equation*}
If there is a node with multiplcity $\ge 2$ then we assume $\mu_1\geq 2$.  We choose $\Omega_1$ by an extreme BQD of order $\mu_1$ and inscribe it in $(\clos\Omega_n)^c$ as in the third configuration of Figure \ref{fig:initialUQD}.
If $d=3$, the resulting quadrature domain after a slight Hele-Shaw flow has connectiviy $\mu_1+1$. 
If $d\ge 4$ we proceed, as in the case with only finite nodes, to inscribe further extreme BQDs.  The connectivity of the final quadrature domain is
\begin{equation*}
  \mu_1+1+ \sum_{j=2}^{n-1}(\mu_j+1)= d+n-2.
\end{equation*}
We obtain a desired quadrature domain after a small Hele-Shaw flow.

\subsection{Bounded quadrature domains}

We construct bounded quadrature domain $\Omega$ to prove Theorem \ref{thm:BQD}. 
We show the construction {\em up to a small Hele-Shaw perturbation}.

\begin{figure}[ht]\begin{center}
\includegraphics[width=0.6\textwidth]{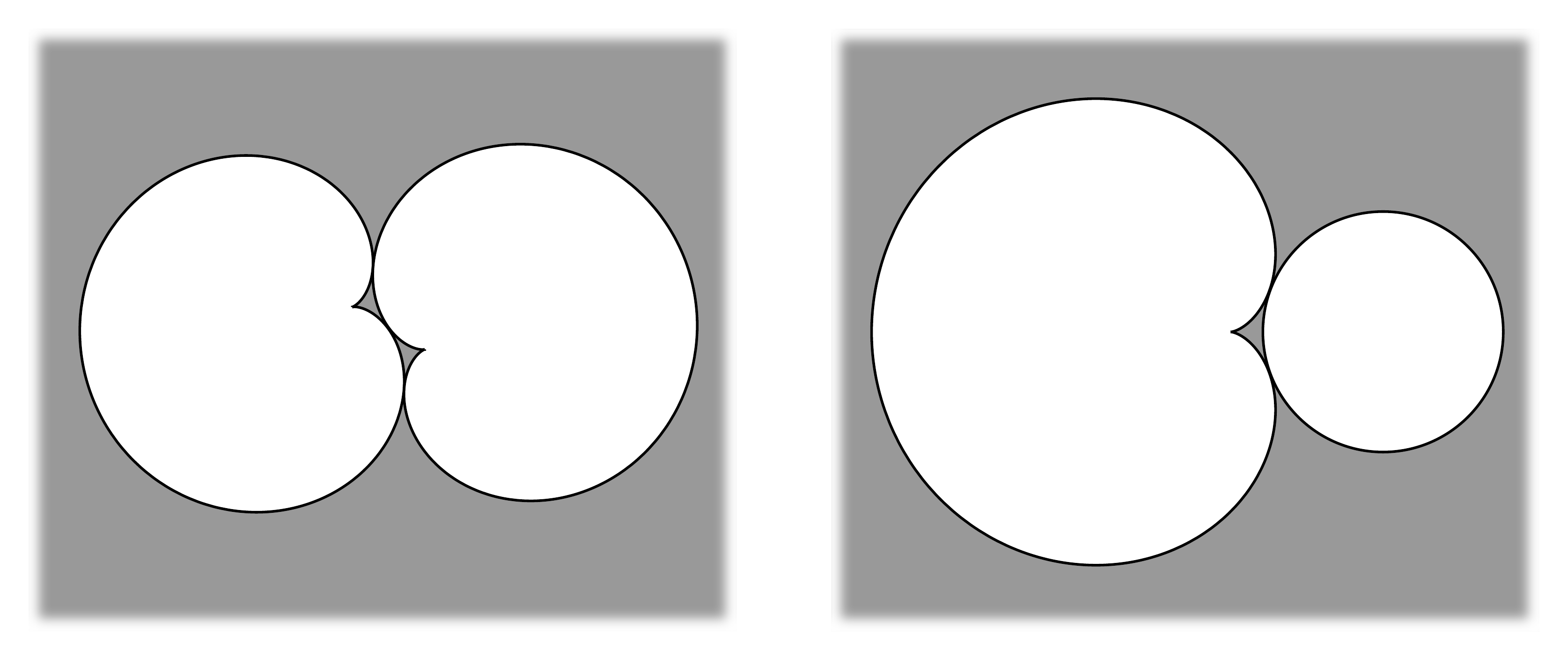}
\caption{\label{fig:initial-config}Initial configurations for BQD}
\end{center}\end{figure}

If all the nodes have multiplicities $\leq 2$ then we have three possibilities.  There are only simple nodes, only double nodes, or mixture of the two.  The first case was considered in \cite{Gus05} (also see Lemma 5.1 in \cite{1stpaper}). 
The initial configurations for the other two cases are shown in Figure \ref{fig:initial-config}.  

For the case with only double nodes,
we consider the first configuration of Figure \ref{fig:initial-config}. There are two cardioids, $\Omega_1$ and $\Omega_2$, such that the connectivity of the interior of $(\Omega_1\sqcup\Omega_2)^c$ is $3$.  (We are done if $n=2$.)
Successive inscription process, that is described in the previous subsection, gives the number of components in the interior of $(\sqcup_{j=1}^n\Omega_j)^c$ by
\begin{equation*}
  3+\sum_{j=2}^n (\mu_j+1) = d+n-3<2d-4.
\end{equation*}

For the case with mixed (simple and double) nodes,
we consider the second configuration of Figure \ref{fig:initial-config}. We have a cardioid and a disk, $\Omega_1$ and $\Omega_2$, such that the connectivity of the interior of $(\Omega_1\sqcup\Omega_2)^c$ is $2$.  (We are done if $n=2$.)
Successive inscription process, that is described in the previous subsection, gives the number of components in the interior of $(\sqcup_{j=1}^n\Omega_j)^c$ by
\begin{equation*}
  2+\sum_{j=2}^n (\mu_j+1) = d+n-3\le 2d-4.
\end{equation*}

If there exists a node of multiplicity $\mu_1\geq 3$, then we start with an extreme BQD $\Omega_1$ of order $\mu_1$.  This contains $\mu_1-2$ deltoid-like curve.  The closure of this domain has the connectivity
 $$\mu_1-1.$$
If $n=1$, we are done.  If $n>1$, we use the same process that is described in the previous subsection to find $\Omega_2,\cdots,\Omega_n$ such that the number of components in the interior of $(\Omega_1\sqcup\cdots\sqcup\Omega_n)^c$ is given by
\begin{equation*}
  \mu_1-1+\sum_{j=2}^n (\mu_j+1) = d+n-2.
\end{equation*}
A small Hele-Shaw perturbation a quadrature domain with the desired property.

\section{Equation \texorpdfstring{$r(z)=\overline z$}{}}\label{sec:rz}

\bs\subsection{Lefschetz fixed point theorem} 

Let us introduce the following notation: if $r$ is a rational function, then
\begin{itemize}
\item[--] $\widehat F_r$ (resp. $F_r$) is the number of fixed points of $\overline r$ in $\widehat\CC$ (resp. in $\CC$).
\item[--] $\widehat A_r$ (resp. $A_r$) is the number of {\em attracting} fixed points of $\overline r$ in $\widehat\CC$ (resp. in $\CC$). \end{itemize}
If $z\in\CC$ and $r(z)=\bar z$, then $z$ is {\it attracting} if $|r'(z)|<1$.  If $r(\infty)=\infty$, then $\infty$ is attracting if $|r'(z)|>1$ where 
$r'(\infty)$ is the limit of $r'(z)$ at $\infty$. A 
 fixed points of $\bar r$ is {\it hyperbolic} if it is either attracting or repelling.

Let $r$ be a rational function of degree $d$. We will assume that the fixed points of $\overline r$ are hyperbolic,
i.e. $$|r'(z)|\ne 1\quad\text{if}\quad \overline{r(z)}=z.$$

\begin{lemma} If $r$ is a rational function of degree $d$ such that all fixed points of $\bar r$ are hyperbolic, then
$$\widehat F_r=2\,\widehat A_r+d-1.$$
\end{lemma}
\begin{proof} 
By Lefschetz formula we have
$$\sum_{a\in {\rm Fix}(\overline r)} {\rm ind}_a(\bar r)=1-0-d,$$
where ${\rm ind}_a({\bar r})$ is the winding number of $\bar r(z)-z$ along the boundary of  a small disc centered at $a$.  The right hand side is the sum of traces of the induced map on the $k$-dimensional homologies.

It is clear that  the index is +1 if $a$ is attracting and $-1$ if repelling. Since the number of repelling fixed points are given by $\widehat F_r-\widehat A_r$ the above equation gives 
$\widehat A_r-(\widehat F_r-\widehat A_r)=1-d$.
\end{proof}

{\bf Remark.} In \cite{Kha1},\cite{Kha2} this lemma was derived from the harmonic argument principle.  While the usual argument principle for holomorphic functions counts the number of zeros, the harmonic argument principle counts the two types of zeros (of the harmonic function $r(z)-\overline z$ in our case) weighted by opposite signs; one corresponds to the attracting fixed points of $\overline r$ and the other to the repelling fixed points.   The relation to the Lefschetz fixed point theorem was mentioned in \cite{Ge03}.

\subsection{Fixed points and unbounded quadrature domains}\label{sec:fixed points-proof}

\begin{thm}\label{thm:62} Let $r$ be a rational function, and $c=A_r$. Then there exists an UQD with quadrature function $r$ such that its connectivity is $\ge c$ .\end{thm}

\begin{proof} Let $z_1,\dots, z_c\in\CC$ be the attracting fixed points of $\bar r$. Denote
$$U=\bigcup_{j=1}^c B(z_j,\delta),$$
where $\delta$ is small enough so the discs are disjoint and do not contain any pole of $r$. We define the algebraic Hele-Shaw potential
$Q:U\to\RR$ as follows:
$$Q(z)=|z|^2-|z_j|^2-2\Re\int_{z_j}^z r(\zeta)~d\zeta,\qquad z\in B(z_j,\delta).$$
The potential $Q$ has a strict local minimum at each point $z_j$ because
$$\partial Q(z_j)=\bar z_j-r(z_j)=0,$$
and
$$\frac14\left|\begin{matrix}Q_{xx}&Q_{xy}\\Q_{xy}&Q_{yy}\end{matrix}\right|_{z=z_j}=(\partial\bar\partial Q)^2-|\partial ^2Q|^2=1-|r'(z_j)|^2>0.$$
Since $Q(z_j)=0$ for all $j$'s, we can assume that $Q\ge0$ in $U$ by decreasing $\delta$ if necessary. Using standard properties of the Hele-Shaw flow, see Section 5.1 of \cite{1stpaper}, we conclude that there is a local droplet of $Q$ which contains all fixed points $z_j$'s. The the unbounded complementary component of the droplet is then a quadrature domain of connectivity $\ge c$.
\end{proof}

The converse statement is implicit in \cite{1stpaper}.
\bs 

\begin{thm}\label{thm:63} Let $\Omega$ be a UQD of connectivity $c$.   Then there exists a rational function $r$ which has the same pole multiplicity structure as the quadrature function of $\Omega$ and such that the number of finite attracting fixed points of $\overline r$ is $\geq c$.\end{thm}

\begin{proof} 
Let $S$ be the Schwarz function of $\Omega$. We first assume that there are no cusps or double points on the boundary, and that
$\overline S$ has no critical values on $\partial\Omega$.  By Lemma 4.3 in \cite{1stpaper}, there exists a rational function $r:\hat\CC\to\hat\CC$ such that $\bar r$ is quasi-conformally conjugate to a certain extension $h$ of 
$$\overline S:~ \Omega\setminus\overline S^{-1}(\Omega^c) \to \widehat\CC$$
to the whole  Riemann sphere. 
We can choose the conjugacy so that $\infty\mapsto\infty$.
The extended function $h$ has an attracting fixed point in each component of $\Omega^c$. 
Since $\overline S$ and $h$ have the same poles, and these poles are in $\Omega\setminus\overline S^{-1}(\Omega^c)$, the functions $\overline r$ and $h$ have the same pole multiplicity structures. Finally, it is explained in Section 5 in \cite{1stpaper} how to remove the assumptions that we made at the beginning of the proof.
\end{proof}

Combining Theorem \ref{thm:62} with connectivity bounds \eqref{eq-thma1}-\eqref{eq-thma2} we obtain the following generalization of
the estimates in \cite{Kha1} and \cite{Kha2}.
\bigskip

\begin{customthm}{\!\!} Let $r$ be a rational function of degree $d\ge2 $ with  $n$  distinct poles.  
\begin{equation}\label{eq:Fr}
  \widehat F_r\le
\min\{3d+2n-3, 5d-5\}.
\end{equation}
\end{customthm}

\begin{proof} The same argument as in \cite{Kha1}, \cite{Kha2} (see e.g. Lemma 1 in \cite{Kha2}) shows that we can assume that all fixed points of $\bar r$ are hyperbolic, so we have $\widehat F_r = 2\widehat A_r +d-1$.

We need to consider 3 cases: 
\begin{itemize}
\item[(a)] $\overline r(\infty)\neq\infty$.
 $$A_r\le\min\{d+n-1,~ 2d-2\}\quad\Longrightarrow \quad \widehat F_r\le\min\{3d+2n-3,~ 5d-5\}.$$
\item[(b)]
$\infty$ is a repelling fixed point.
 $$A_r\le d+n-2\quad\Longrightarrow \quad \widehat F_r\le 3d+2n-5.$$
\item[(c)] $\infty$ is an attracting fixed point. Here we use
$$A_r\le \min\{d+n-2,2d-3\}\quad\Longrightarrow \quad \widehat F_r\le\min\{3d+2n-3,~ 5d-5\}.$$
The inequality on $A_r$ follows from \eqref{eq-thma2} if $n<d$. If $n=d$ we can not have $A_r=2d-2$ because then we would have
$\widehat A_r=2d-1$ and $\widehat N=5d-3$, which contradicts \cite{Kha2}.
\end{itemize}
\end{proof}

Let us now show that the estimate \eqref{eq:Fr} is sharp.

Proof of Theorem \ref{thm:C}. Let 
$$d=\mu_1+ ...+\mu_n$$
be a given partition.  For the first part of the theorem, we need to find a rational function $r$ satisfying
\begin{itemize}
\item[(i)] $r(\infty)\ne\infty$,

\item[(ii)] $A_r=\min\{d+n-1, 2d-2\}$,

\item[(iii)] all fixed points are hyperbolic. 
\end{itemize}
The existence of such a function follows from the construction of UQDs in the proof of Theorem \ref{thm:UQD} and from the
 argument in the proof of Theorem \ref{thm:63}. The important feature of the construction is that all the critical points of $r$ are attracted to attracting fixed points of $\bar r$, which excludes the possibility of non-hypebolic fixed points -- such points would be {\it parabolic} for the second iterate of $\bar r$.

The same reasoning proves the second part of Theorem \ref{thm:C} where we use our construction of UQDs such that
$\infty$ is a node of multiplicity $\mu_n>1$ and 
$$\conn \Omega=d+n-2=\min \{d+n-2,2d-3\}.$$

{\bf Remark.} One can show that for all $d\ge2$ there are UQDs with $n=d$ such that $\infty$ is an attracting fixedpoint of the quadrature function. (The proof is computer assisted, and we don't present it here.) It follows that the second part of Theorem \ref{thm:C} is also true for $\mu_n=1$

\subsection{Examples in gravitational lensing}\label{sec:grav}

In this subsection we will consider the equation $r(z)=\bar z$ with 
\begin{equation}\label{eq:r-gravi}
  r(z)= -\gamma z + \sum_{j=1}^N \frac{\varepsilon_j}{z-z_j}
\end{equation}
where
\begin{equation}\label{eq:phys-constraint1}   \varepsilon_j>0,~~~ \gamma\ge0, ~~~  s\in\CC,~~~ z_j\in\CC.
\end{equation}
Let $F_r$ be the number of finite solutions (``the number of images'', see Section \ref{sec:GL-intro}).   By \cite{Kha2} we have
$$F_r\le 5N-1,$$
and by Theorem \ref{thm:C} this estimate is sharp if we don't require \eqref{eq:phys-constraint1}. Our goal is to construct examples with $5N-1$ images for functions  $r(z)$ satisfying \eqref{eq:phys-constraint1}.

\bigskip

\begin{prop}\label{thm:smallgamma} If $N>0$ is an even number and $\gamma\in (\gamma_0,1)$ where
\begin{equation*}
  \gamma_0 = \frac{2-\sqrt2}{2+\sqrt 2}\approx 0.171573,
\end{equation*}
then there exist $\{\varepsilon_j|j=1,\cdots,N\}\subset \RR^+$ and $\{z_j|j=1,\cdots,N\}\subset \CC$ such that $r(z)=\overline z$ with $r(z)$ given by \eqref{eq:r-gravi} has $5N-1$ roots.
\end{prop}

\begin{proof}
An ellipse is the complement of the conformal image of the exterior unit disk under the map:
\begin{equation*}
  f(w)=a w- b/w,\qquad a>b>0.
\end{equation*}
Since the Schwarz function of the complement of the ellipse satisfies $S(f(w))=\overline{f(1/\overline w)}$, the leading behavior of $S(z)$ as $|z|\to\infty$ is given by
\begin{equation}
  S(z)=-\gamma z +{\mathcal O}(1),\quad \gamma= \frac{b}{a}.
\end{equation}
The ellipse has $2(a-b)$ and $2(a+b)$ as the minor (horizontal) and major (vertical) axis respectively.   The ellipse is a disk when $\gamma=0$, and gets more skinny as $\gamma\to 1$.

We choose $b=2-a$ so that the major axis has {\em a fixed length $4$}.  As $a$ decreases from $2$ to $1$ (i.e. $\gamma=2/a - 1$ increases from zero to one), there exists a critical $a$ such that the ellipse begins to contain the two unit disks along the major axis (as shown in the left configuration of Figure \ref{fig:goggle}).   This happens when
\begin{equation}
  a=\frac{2+\sqrt 2}{2},\quad b=\frac{2-\sqrt 2}{2},\quad \gamma=\gamma_0:=\frac{2-\sqrt 2}{2+\sqrt 2}.
\end{equation}

\begin{figure}\begin{center}
\includegraphics[width=0.35\textwidth,angle=90]{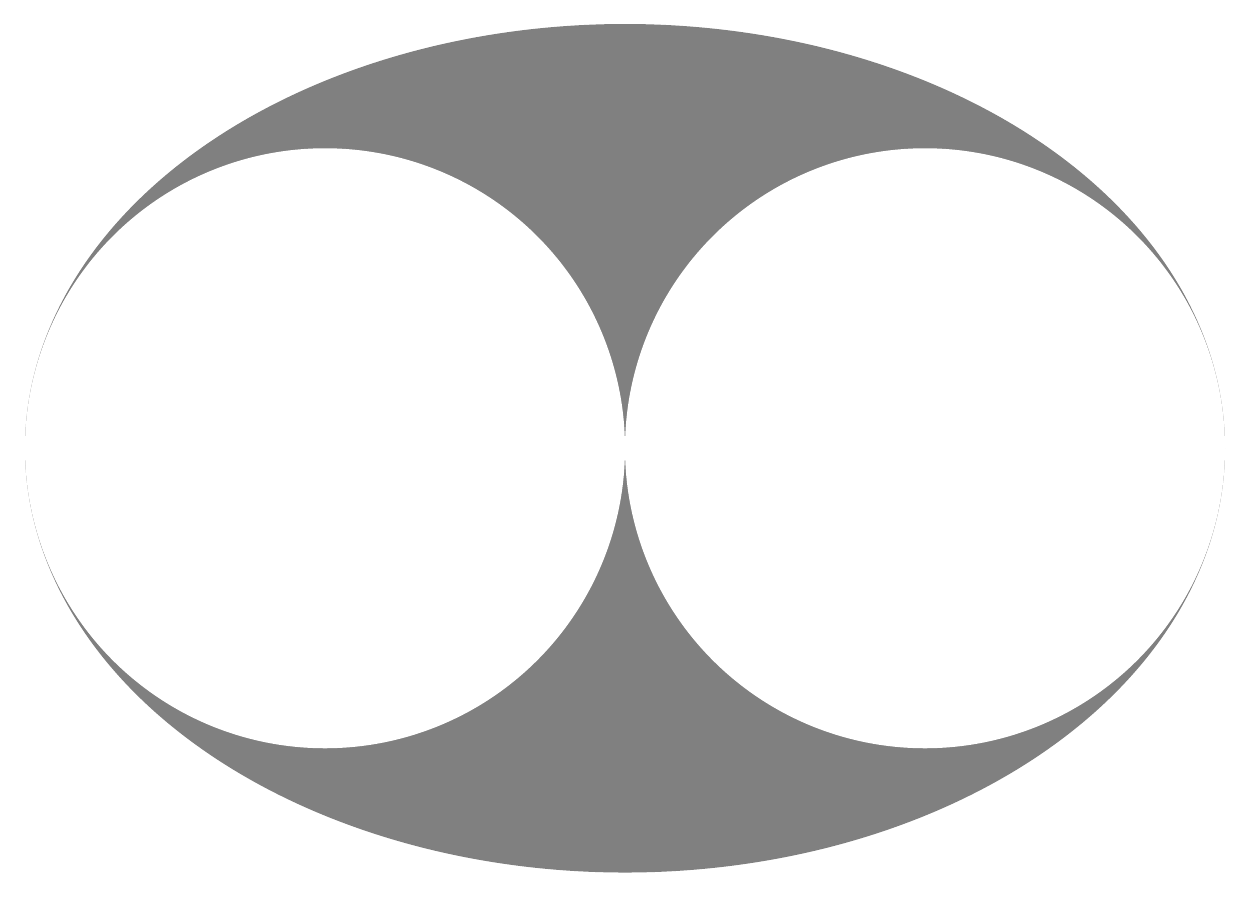}\qquad
\includegraphics[width=0.37\textwidth,angle=90]{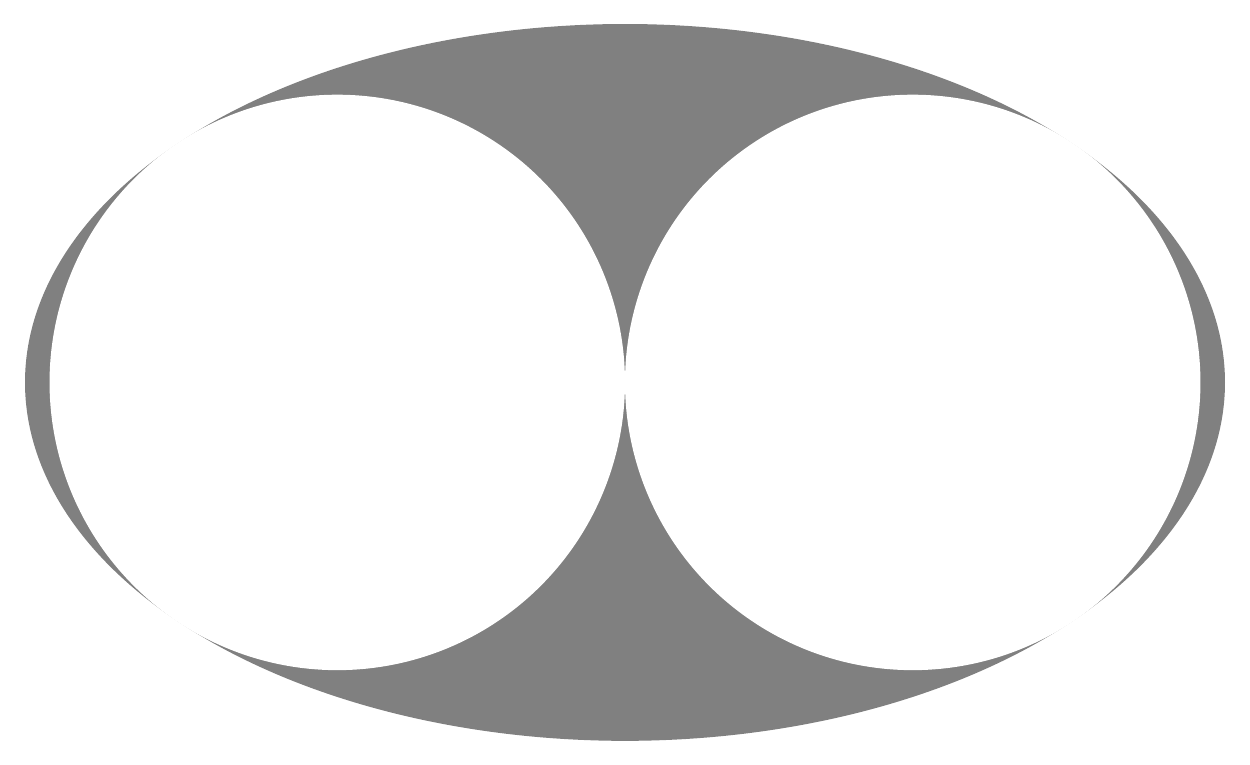}
\label{fig:goggle}
\caption{Left: The ellipse with $\gamma=\gamma_0$; Right: Packing of two disks inside an ellipse with $\gamma>\gamma_0$.}
\end{center}
\end{figure}

For any $\gamma\in (\gamma_0,1)$ one can inscribe two disks of identical radius inside the ellipse so that there are four components in the complement of (closure of) the two disks, see the right configuration in Figure \ref{fig:goggle}. 

Among the four components there are two components with deltoid-like boundary, and we can inscribe $N-2$ disks using Theorem \ref{thm:circ-in-delt} {\em along the minor axis symmetrically with respect to the major axis}.  
We then have $2N$ components with $N$ (which is even) disks packed inside the ellipse.  

The resulting Schwarz function $S$ has $N$ simple poles and all of them have positive residues (given by the square of the corresponding radius).  Two of them are on the vertical axis, and the others are all on the horizontal axis.

We then deform the droplet by a small Hele-Shaw flow followed by a quasi-conformal surgery (which is explained in details in \cite{1stpaper}, Section 4.3).  It results in the existence of a quasi-conformal homeomorphism $\Phi:\widehat\CC\to\widehat\CC$ satisfying $\Phi:\infty\mapsto\infty$ and $\Phi'(\infty)>0$ such that the anti-analytic continuation of 
$\Phi \circ \overline{S}\circ \Phi^{-1}$ defines an anti-analytic rational function, say $\overline r$, with $2N$ attracting fixed points.  Let $z_j$'s be the poles of $S$ (i.e. the centers of the disks), $\Phi$ is holomorphic at $z_j$ and symmetric, i.e. mapping $\RR$ to $\RR$ and $\ii\RR$ to $\ii\RR$ while preserving orientation.   This gives $\Phi'(z_j)>0$.   
Then $r$ has a simple pole at $\Phi(z_j)$, and the residue of $r$ at $\Phi(z_j)$ is given by 
\begin{equation*}
   \varepsilon_j:=\frac{\overline{\Phi'(\infty)}}{\Phi'(z_j)}\text{Res}_{z_j}S.
\end{equation*} 
Therefore we have $\varepsilon_j>0$. 
\end{proof}

The similar procedure gives the following result which is valid for all $N$.
\bigskip
\begin{prop}\label{cor:gravi} Given $N>0$, there exists a rational function $r$ of the form \eqref{eq:r-gravi} with $0<\gamma<1$, $\varepsilon_j>0$, such that $r(z)=\overline z$ has exactly $5N-1$ finite roots.
\end{prop}

\begin{proof} In this case we choose a very skinny ellipse and inscribe $N$ disjoint disks so that i) each disk has three tangent points (with other discs or the ellipse) ii) and the centers of the discs are on the major axis of the ellipse. 
\end{proof}

{\bf Remark.} Sharp examples for case $\gamma=0$ can be obtained from ${\mathbb Z}_{N-1}$-symmetric packing of $N-1$ discs inside a round annulus, \cite{Rh03}.

\subsection{Rational analogues of Crofoot-Sarason polynomials}\label{sec:Crofoot-Sarason}

Let us recall Geyer's theorem \cite{Ge03}:  {\em For every $d > 1$, there exists a polynomial $p$ of degree $d$ and distinct points $c_1,...,c_{d-1}$ in $\CC$ such that}
\begin{equation*}
  \overline{p(c_j)}=c_j,\qquad p'(c_j)=0.
\end{equation*}

{\bf Example} (Crafoot-Sarason, see \cite{Kha1}):
\begin{equation*}
  p(z)=\frac{3z}{2}-\frac{z^3}{2}.
\end{equation*}
This satisfies $p'(\pm 1)=0$ and $p(\pm1)=\pm1$.

Here is a counterpart of Geyer's theorem for rational functions.  
\vspace{0.2cm}
\begin{thm}\label{thm:Rhie} For every $d>1$, there exists a rational function $r$ of degree $d$ and distinct points $c_1,\cdots,c_{2d-2}$ in $\CC$ such that
\begin{equation*}
	r'(c_j)=0,\qquad r(c_j)=\overline{c_j}.
\end{equation*}
\end{thm}
\begin{proof}
Circle packing gives us a droplet with $2d-2$ circular-triangles, and our surgery procedure results in the same number of finite critical fixed points.  
\end{proof}
{\bf Examples.}
\begin{itemize}
\item[(i)] If $d=n=2$ and $\infty$ is {\em not} a pole, (the complex conjugate of) the rational function,
\begin{equation*}
	r(z)=\frac{2z}{z^2-1},
\end{equation*}
 has $2d-2=2$ critical fixed points, i.e. $r'(\pm i)=0$ and $\overline{r(\pm i)}=\pm i$.
\item[(ii)] If $d=n=2$ and infinity {\em is} a pole, the following function has $2d-2=2$ critical fixed points:
\begin{equation*}
 	r(z)=\frac{z}{2}+\frac{1}{z}.
 \end{equation*} 
 We have $r'(\pm\sqrt 2)=0$ and $r(\pm\sqrt2)=\pm\sqrt2$.
\end{itemize}

We can ``interpolate'' between Geyer's theorem and \ref{thm:Rhie}.  Recall the sharpness results for the number of attracting fixed points.
\begin{itemize}
\item[(I)] Given numbers $d\geq 2$, $n\leq d$, and a partition $d=\mu_1+\cdots+\mu_n$ (where $\mu_j$ are positive integers), there exists a rational function $r$ of degree $d$ with $n$ finite poles of multiplicities $\mu_1,\cdots,\mu_n$ such that $\overline r$ has
\begin{equation*}
	\min(d+n-1,2d-2)
\end{equation*}
attracting fixed points.
\item[(II)] Also, there exists a rational function $r$ of order $d$ such that infinity is a pole of multiplicity $\mu_n$, and  $\mu_1\cdots,\mu_{n-1}$ are the multiplcities of the finite poles, and such that $\overline r$ has $d+n-2$ finite attracting fixed points.
\end{itemize}
It turns out that we can replace ``attracting'' with ``super-attracting'' in part (II).

\vspace{0.2cm}
\begin{thm}\label{thm:super} 
Given numbers $d\geq 2$, $n\leq d$, and a partition $d=\mu_1+\cdots+\mu_n$, there exists a rational function $r$ of degree $d$ such that infinity is a pole of multiplicity $\mu_n$, and $\mu_1\cdots,\mu_{n-1}$ are the multiplcities of the finite poles, and such that $\overline r$ has $d+n-2$ finite \underline{super}-attracting (or critical) fixed points.
\end{thm}

\begin{proof} If $\mu_n\geq 2$ then the droplets in the proof of (the second part of) Theorem \ref{thm:UQD} has $d-n-2$ deltoid-like curves as its boundary (see Section \ref{sec:main-proof}), and our surgery procedure results in the same number of finite critical fixed points.  If $\mu_n=1$ then we use a very skinny ellipse so the bygones contain critical values of the Schwarz reflection in the ellipse. The surgery again gives the desired number of critical fixed points.
\end{proof}

Combining Theorem \ref{thm:Rhie} and \ref{thm:super} we get
\vspace{0.2cm}
\begin{cor} 
For $d\geq 2$, $n\leq d$ and a partition $d=\mu_1+\cdots+\mu_n$ there exists a rational function with $d+n-2$ finite critical fixed points.  Moreover we can prescribe the multiplicities of the pole at $\infty$. 
\end{cor}

{\bf Examples.}
\begin{itemize}
\item[(i)] For the rational function with $d=2$ and $n=1$ given by
\begin{equation*}
r(z)= \frac{1}{z^2}+\frac{a}{z}+\frac{a^2}{4}-\frac{2}{\overline a},\quad a\neq 0,
\end{equation*}
there is at most one ($d+n-2=1$) critical fixed point given by
$r'(-2/a)=0$ and $r(-2/a)=-2/\overline a$. Even though there are two ($2d-2=2$) critical points, one of them is at the pole and, therefore, cannot be a fixed point.   The rational function can still have two fixed points: for $a>2^{5/3}/3^{1/3}$, there is another attracting (non-critical) fixed point at $z=(a^2 + \sqrt{32 a + a^4})/8$.
\item[(ii)] We have $d=3$, $n=2$ in the case
$$r(z)=\frac{z^2}{2}+\frac{c^3}{z}, \qquad c=\frac{2}{3}.$$
The points
$$c_j=\exp(2\pi i j / 3)\cdot c,\qquad j=1,2,3.$$
are critical fixed points of $\overline r$.  Note $d+n-2=3+2-2=3$.  Compare with an Apollonian packing of a disk inside the deltoid.
\item[(iii)]
A generalization:
\begin{equation*}
  r(z) = \frac{z^{d-1}}{d-1}+\frac{c^d}{z}, \quad c=\bigg(\frac{d-1}d\bigg)^{1/(d-2)}.
\end{equation*}
Then $\{\sqrt[d]{1} \cdot c\}$ are the critical fixed points of $\overline r$.  
\end{itemize}

{\bf Remark.}  We cannot replace ``attracting'' with ``super-attracting'' in part (I) unless $n=d$ (Theorem \ref{thm:Rhie}).   Though there are $2d-2$ critical points (counted with multiplicities), 
$$d-n=(\mu_1-1)+\cdots+(\mu_n-1)$$
of them are poles.  Poles are not fixed points since $\infty$ is not a pole.  Thus we are left with only $d+n-2$ critical points, which is strictly less than $\min(d+n-1,2d-2)$ unless $n=d$.

\bibliographystyle{alpha}

\vspace{1cm}
{\sc Department of Mathematics and Statistics, University of South Florida,\newline Tampa, FL 33647, USA}
\\
{\it E-mail address:} {\tt lees3@usf.edu}
\vspace{1cm}

{\sc Department of Mathematics, California Institute of Technology,\newline Pasadena, CA 91125, USA}
\\
{\it E-mail address:} {\tt makarov@caltech.edu}

\end{document}